\documentclass[11pt,reqno]{amsart} 

\usepackage[dvipsnames]{xcolor}
\usepackage{tikz}
\usepackage{stmaryrd}
\usetikzlibrary{matrix,arrows,decorations.pathmorphing}
\usepackage{graphicx}
\usepackage{setspace}
\usepackage{amssymb}
\usepackage[stable]{footmisc}
\usepackage{amsmath,mathtools}
\usepackage{fullpage}
\usepackage{caption}
\usepackage{amsthm}
\usepackage{mathrsfs}
\usepackage{thmtools}
\usepackage{blkarray}
\usepackage{multirow}
\usepackage{picinpar} 
\usepackage{tikz-cd}
\usepackage{color}
\usepackage{verbatim}
\usepackage{hyperref}
\hypersetup{colorlinks=true, linkcolor=black, citecolor = OliveGreen, urlcolor = magenta}
\usepackage{amssymb}
\usepackage{amsmath}
\usepackage{enumitem,colonequals}
\usepackage{color}
\usepackage{mathrsfs}
\usepackage[all]{xy}
\usepackage{comment}
\usepackage[margin=1.in]{geometry}

\newcommand{\mZ}{\mathbb{Z}}
\newcommand{\mR}{\mathbb{R}}
\newcommand{\mQ}{\mathbb{Q}}
\newcommand{\mF}{\mathbb{F}}
\newcommand{\mC}{\mathbb{C}}

\newcommand{\cC}{\mathcal{C}}
\newcommand{\cX}{\mathcal{X}}
\newcommand{\cI}{\mathcal{I}}
\newcommand{\cY}{\mathcal{Y}}
\newcommand{\cZ}{\mathcal{Z}}
\newcommand{\cE}{\mathcal{E}}
\newcommand{\cD}{\mathcal{D}}

\newcommand{\cS}{\mathcal{S}}
\newcommand{\cO}{\mathcal{O}}
\newcommand{\cU}{\mathcal{U}}

\newcommand{\cB}{\mathcal{B}}
\newcommand{\cM}{\mathcal{M}}
\newcommand{\wt}[1]{\widetilde{#1}}

\newcommand{\cL}{\mathcal{L}}

\usepackage[OT2,T1]{fontenc}
\DeclareSymbolFont{cyrletters}{OT2}{wncyr}{m}{n}
\DeclareMathSymbol{\Sha}{\mathalpha}{cyrletters}{"58}

\DeclareMathSymbol{\Sha}{\mathalpha}{cyrletters}{"58}

\newcommand{\nrm}[1]{\left|\left |#1\right |\right |}

\newcommand{\brk}[1]{ \left\lbrace #1 \right\rbrace }
\newcommand{\pwr}[1]{ \left( #1 \right) }

\newcommand{\CITE}{{\color{red}\textsf{CITE}}}

\newcommand{\sH}{{\mathscr H}}

\newcommand{\sL}{{\mathscr L}}
\newcommand{\sM}{{\mathscr M}}

\newcommand{\jax}[1]{{\color{cyan} \sf $\clubsuit\clubsuit\clubsuit$ Jackson: [#1] $\clubsuit\clubsuit\clubsuit$}}

\RequirePackage{mathrsfs} 

\theoremstyle{theorem}
\numberwithin{equation}{subsection}
\newtheorem{thmx}{\text{Theorem}}

\newtheorem{theorem}[subsubsection]{Theorem}

\newtheorem{lemma}[subsubsection]{Lemma}

\newtheorem{corollary}[subsubsection]{Corollary}

\newtheorem{prop}[subsubsection]{Proposition}

\numberwithin{equation}{subsection}

\theoremstyle{definition}
\newtheorem{definition}[subsubsection]{\text{Definition}}
\newtheorem{notation}[subsubsection]{\text{Notation}}

\newtheorem{remark}[subsubsection]{Remark}

\newtheorem{example}[subsubsection]{Example}

\theoremstyle{remark}

\numberwithin{equation}{subsubsection} \numberwithin{figure}{section}

\DeclareMathOperator{\an}{an}

\DeclareMathOperator{\Pic}{Pic}

\DeclareMathOperator{\aPic}{\widehat{Pic}}
\DeclareMathOperator{\acPic}{\widehat{\underline{Pic}}}
\DeclareMathOperator{\ddiv}{div}
\DeclareMathOperator{\Div}{Div}

\DeclareMathOperator{\aDiv}{\widehat{Div}}
\DeclareMathOperator{\aCaCl}{\widehat{CaCl}}
\DeclareMathOperator{\aPr}{\widehat{Pr}}

\DeclareMathOperator{\Gal}{Gal} 
 \DeclareMathOperator{\Spec}{Spec}

\DeclareMathOperator{\Hom}{Hom}

\DeclareMathOperator{\red}{red}

\newcommand{\cdef}[1]{\textsf{\textit{#1}}}

\renewcommand{\leq}{\leqslant}

\renewcommand{\geq}{\geqslant}

\DeclareMathOperator{\FS}{FS}
\DeclareMathOperator{\tFS}{FS^{\tau}}

\DeclareMathOperator{\gtFS}{g.FS^{\tau}}
\DeclareMathOperator{\Ban}{Ban}

\DeclareMathOperator{\PSH}{PSH}

\DeclareMathOperator{\semip}{sp}
\DeclareMathOperator{\gsemip}{g.sp}
\DeclareMathOperator{\eqv}{eqv}
\DeclareMathOperator{\YZ}{YZ}
\DeclareMathOperator{\cptf}{cptf}
\DeclareMathOperator{\model}{mod}

\DeclareMathOperator{\ssa}{s.sa}

\makeatletter
\@namedef{subjclassname@1991}{\emph{2020} Mathematics Subject Classification}
\makeatother

\begin{document}

\title{Global pluripotential theory for adelic line bundles}

\author{Jackson S. Morrow}
\address{Jackson S. Morrow \\
	Department of Mathematics\\
	University of North Texas \\
	Denton, TX 76203, USA}
\email{jackson.morrow@unt.edu}

\begin{abstract}
In this work, we relate recent work of Yuan--Zhang and Song on adelic line bundles over quasi-projective arithmetic varieties to recent advances in pluripotential theory on global Berkovich spaces from Pille-Schneider.

In particular, we establish an equivalence between subcategories of adelic line bundles on quasi-projective varieties and line bundles on their Berkovich analytifications equipped with a continuous plurisubharmonic metric. 
We also provide several applications of this equivalence.

For example, we generalize a construction of Pille-Schneider concerning families of Monge--Amp\`ere measures on analytifications of projective arithmetic varieties to the quasi-projective setting. 
With this construction, we offer a new description of non-degenerate subvarieties which involves Monge--Amp\`ere measures over trivially valued fields. 
Finally, we define a Monge--Amp\`ere measure on the analytification of a quasi-projective arithmetic variety. 
\end{abstract}


\date{\today}
\maketitle

\tableofcontents

\section{Introduction}
In this work, an arithmetic variety will refer to a separated, integral scheme of finite type over $\Spec(\mZ)$ such that the structure morphism is flat.  
The purpose of this work is to establish a connection between adelic line bundles on quasi-projective arithmetic varieties introduced by Yuan and Zhang \cite{YuanZhang:AdelicLineBundles} and pluripotential theory on Berkovich spaces over general Banach rings developed by Pille-Schneider \cite{PilleSchneider:Global}.

\subsection*{Motivation}
Heights of algebraic points on varieties play a fundamental role in Diophantine geometry as evident through the proofs of the Mordell conjecture by Faltings \cite{Faltings2} and Vojta  \cite{Vojta:Mordell} and the proofs of Bogomolov conjecture by Ullmo \cite{Ullmo:PositivityPoints} and Zhang \cite{Zhang:EquidistributionSmallPoints}.
For projective arithmetic varieties, Zhang \cite{Zhang:SmallPoints} introduced the notion of an adelic line bundle and used these objects in his proof of the Bogomolov conjecture for abelian varieties. 
Recently, Yuan and Zhang \cite{YuanZhang:AdelicLineBundles} extended Zhang's construction of adelic line bundles to quasi-projective arithmetic varieties. 
These objects have already appeared in several recent works concerning heights of algebraic points and arithmetic dynamics, for example \cite{Yuan:ArithmeticBignessUniformBogomolov, GaoZhang:HeightsPeriodsAlgebraicCyclesFamiles, DemarcoMavraki:DynamicsP1, DemarcoMavraki:GeometryPreperiodicPointsFamilies, MavrakiSchmidt:DynamicalBogomolovFamiliesSplitRational,  JiXie:DAOCurves, Abbound:UnlikelyIntersections}. 

Roughly speaking, an adelic line bundle on a quasi-projective arithmetic variety $\cU$ is a Cauchy sequence of line bundles on projective models $\cX$ of $\cU$ where the metric on this space of model line bundles is induced by a so-called boundary divisor. 
We refer the reader to Section \ref{sec:adelicdivisorsandlinebundles} for details. 
In addition to defining adelic line bundles on quasi-projective arithmetic varieties, Yuan and Zhang \cite[Section 4]{YuanZhang:AdelicLineBundles} also constructed an absolute intersection pairing and a relative intersection pairing using the Deligne pairing, which generalized the classical intersection pairings of Gillet--Soul\'e \cite{GilletSoule:ArithmeticIntersectionTheory, GilletSoule:ArithmeticRiemannRoch}. These results allowed them to establish a notion of height and volume for these adelic line bundles, which lead them to a height inequality and a general equidistribution result \cite[Theorem 5.3.5 \& Theorem 5.4.3]{YuanZhang:AdelicLineBundles}, respectively.  
Their intersection pairing is defined on a class of adelic line bundles, the so-called integrable adelic line bundles. 
Roughly, an adelic line bundle is integrable if it is isometric to a difference of adelic line bundles where each adelic line bundle can be represented as a limit of nef models. 
In their context, nef models of line bundles come equipped with a continuous semipositive metric on the complex points of the line bundle and satisfy an Arakelov intersection theoretic criterion for nefness (cf.~\cite[Appendix A.4]{YuanZhang:AdelicLineBundles}).

While their construction is involved, Yuan and Zhang \cite[Section 3.4]{YuanZhang:AdelicLineBundles} showed that adelic line bundles on quasi-projective arithmetic varieties can be interpreted as certain metrized line bundles on the associated Berkovich analytification of $\cU$. 
We refer the reader to Section \ref{sec:Berkovich} for details on Berkovich analytifications over general Banach rings. 
Very recently, Song \cite{Song:EquivariantAdelic} proved that there is in fact a bijection between adelic line bundles on quasi-projective arithmetic varieties and continuous, norm-equivariant metrized line bundles on the associated Berkovich analytification.

This result establishes a firm relationship between adelic line bundles and metrized line bundles on Berkovich spaces and is the starting point for our work.

\subsection*{Statement of results}
In this work, we identify a subcategory of adelic line bundles on quasi-projective arithmetic varieties, which roughly correspond to those that can be represented as a limit of \textit{semiample} models. 
We note that this definition requires no Arakelov condition, but it does impose a stronger global constraint on the model. 
By requiring semiample models, we are able to relate this class of adelic line bundles to metrized line bundles on Berkovich spaces equipped with certain kind of plurisubharmonic metric. 
Pluripotential theory is well-established in the complex setting \cite{BedforTaylor:CapacityPSH, Demailly:Regularization}, and in recent years, it has been extended to the non-Archimedean setting through the works of \cite{thuillier2005theorie, BakerRumley:PotentialTheory, BouksomFJ:SingularSemiPositive, BoucksomJonsson:SingularPSH}. 
Recently, \cite{PilleSchneider:Global} integrated these approaches into a global theory on Berkovich spaces over general Banach rings.

Our main theorem is the following. 

\begin{thmx}\label{thmx:main0}
Let $\cU$ be a flat, quasi-projective, integral scheme of finite type over $\Spec(\mZ)$. 
Then, the category of strongly semiample adelic line bundles on $\cU$ (\autoref{defn:ssalinebundles}) is equivalent to the category of line bundles on $\cU^{\an}$ equipped with a norm-equivariant and continuous semipositive metric (\autoref{defn:continuous_semipositive}). 
\end{thmx}

\subsection*{New ingredients}
Let $\cU$ be a quasi-projective arithmetic variety, and let $\cL$ be a line bundle on $\cU$. 
The main new aspect of this work is our definition of a continuous semipositive metric $\phi$ on $\cL$ (\autoref{defn:continuous_semipositive}). 
This definition involves the data of a net $(\cX_i,\cL_i)$ of projective models for the pair $(\cU,\cL)$ where each $\cL_i$ is equipped with a global tropical Fubini--Study metric $\phi_i$ (\autoref{defn:tropicalFS}) such that the net $\phi_i$ compactly converges to $\phi$ on $\cU^{\an}$. 
When $\cU$ is projective, we show in \autoref{lemma:proj_continuoussemipositive_continuouspsh} that a continuous semipositive metric is the same thing as a continouous plurisubharmonic metric, which was previously defined by \cite{PilleSchneider:Global}. 
To prove \autoref{thmx:main0}, we relate these metrized line bundles to adelic divisors on (quasi-)projective varieties,  and then use ideas from Song's proof of \cite[Theorem 1.1]{Song:EquivariantAdelic} to deduce our result. 

Before discussing applications, we note that our notion of strongly semiample is related to Zhang's notion \cite{Zhang:PositiveArithmeticVarieties} of semiample metrized adelic line bundles. We refer the reader to \autoref{remark:semiamplemetrized} for details.  

\subsection*{Applications}
Using \autoref{thmx:main0}, we derive several applications:
\begin{enumerate}
\item Defining families of Monge--Amp\`ere measures for continuous semipositive metrized line bundles on the analytification of a quasi-projective arithmetic variety (Subsection \ref{subsec:familiesMA}); 
\item Realizing the invariant adelic line bundle on a polarized dynamical system from Yuan--Zhang as line bundle equipped with a continuous semipositive metric (Subsection \ref{subsec:invariant_adelic});
\item Providing a characterization of a non-degenerate subvariety involving Monge--Am\`pere measures over infinite, trivially valued fields (Subsection \ref{subsec:Nondegeneracy_trivially}); 
\item Defining a global Monge--Amp\`ere measure using the family of Monge--Amp\`ere measures for continuous semipositive metrized line bundles (Section \ref{sec:globalMA}). 
\end{enumerate}

The applications described in Section \ref{sec:applications} mainly follow from Yuan and Zhang's proofs in the strongly nef setting. 
That being said, our definitions allows us to define Monge--Amp\`ere measure for the pullback of a continuous semipositive metric along a trivially valued point via the same weak limit process. 
This extra flexibility is crucial in our application to a non-degeneracy criterion over infinite, trivially valued fields. 
Finally, our global Monge--Amp\`ere measure is defined using a classical pullback measure construction where the fiberwise measures are defined via Subsection \ref{subsec:familiesMA} and the measure on the base, the Berkovich analytification of $\mZ$, appears to be new. 

\subsection*{Organization}
In Section \ref{sec:Berkovich}, we recall background on Berkovich analytic spaces over general Banach rings. 
Section \ref{sec:globalpluripotential} discusses background on pluripotential theory on global Berkovich spaces. In this section, we introduce our new notion of continuous semipositive metric (\autoref{defn:continuous_semipositive}). 
We recall background on metrized line bundles and arithmetic divisors on projective arithmetic varieties in Section \ref{sec:metrizedprojectivearithmetic}, and use these notions in Section \ref{sec:adelicdivisorsandlinebundles} when recalling Yuan and Zhang's \cite{YuanZhang:AdelicLineBundles} definition of adelic divisors and line bundles on quasi-projective arithmetic varieties. 
In Section \ref{sec:proofs}, we introduce the subcategory of strongly semiample adelic line bundles on quasi-projective varieties and prove \autoref{thmx:main0}. 
We describe applications of \autoref{thmx:main0} in Section \ref{sec:applications}, and finally, in Section \ref{sec:globalMA}, we define a global variant of the Monge--Amp\`ere measure on a global Berkovich space. 

\subsection{Conventions}
\label{subsec:conventions_AG}
Let $k$ be a field. A \cdef{$k$-variety} is a separated, integral scheme of finite type over $\Spec(k)$. 
Let $A$ be a commutative ring with identity. 
An \cdef{$A$-variety} is a separated, integral scheme of finite type over $\Spec(A)$ such that the structure morphism is flat. 
In the special case when $A = \mZ$, we will refer to an $A$-variety as an \cdef{arithmetic variety}. 
Generally, we will use Roman letters $X,Y,Z$ for varieties defined over fields or general commutative rings and calligraphic letters $\cX,\cY,\cZ$ for arithmetic varieties. We will freely use the notions of projective and quasi-projective arithmetic varieties (more generally, $A$-varieties), and we refer the reader to \cite[\href{https://stacks.math.columbia.edu/tag/01W7}{Tag 01W7}]{stacks-project} and \cite[\href{https://stacks.math.columbia.edu/tag/01VV}{Tag 01VV}]{stacks-project}, respectively, for the definitions and basic properties.

\subsection*{Acknowledgements}
We thank Ziyang Gao and L\'eonard Pille-Schneider for helpful comments on a first draft.  
During preparation of this manuscript, the author was partially supported by NSF DMS-2418796.

\section{Berkovich spaces}
\label{sec:Berkovich}
In this section, we recall aspects of the theory of Berkovich spaces over Banach rings. 

\subsection{Berkovich spaces over commutative Banach rings}
Let $(A,|\cdot|_{\Ban})$ be a commutative Banach ring with identity i.e., a non-zero ring $A$ equipped with a sub-multiplicative norm $|\cdot|_{\Ban}$ such that $A$ is complete with respect to $|\cdot|_{\Ban}$. 
A primary example for our purposes will be $A = \mZ$ endowed with the standard absolute value $|\cdot|$. 
The Berkovich spectrum $\sM(A)$ is the set whose points $x\in \sM(A)$ are multiplicative seminorms $|\cdot|_x\colon A \to \mR_{\geq 0}$ satisfying $|\cdot|_x \leq |\cdot|_{\Ban}$. 
The set $\sM(A)$ is equipped with the topology of pointwise convergence on $A$ which makes it into a non-empty, Hausdorff, compact topological space with a continuous map
\begin{equation}\label{eqn:kernelmap}
\iota \colon \sM(A) \to \Spec(A),\quad |\cdot|_x \mapsto \mathfrak{p}_x = \ker(|\cdot|_x)
\end{equation}
which we will refer to as the \cdef{kernel map}. 
Via this map, we can view $|\cdot|_x$ as a norm on the residue field $\kappa(\iota(x))$. 
We say that the completion $\sH_x$ of the field $(\kappa(\iota(x)),|\cdot|_x)$ is called the \cdef{residue field at $x$}. 
If $\sH_x$ is non-Archimedean, let $R_{x}$ denote the valuation ring of $\sH_x$. 
We refer the reader to \autoref{exam:BerkovichoverZ}, \cite[Example 1.4.1]{BerkovichSpectral}, or \cite[Example 1.3]{PilleSchneider:Global} for a detailed discussion of $\sM(\mZ)$. 

When $R$ is a finitely generated $A$-algebra, the Berkovich analytification $\sM(R)$ of $\Spec(R)$ is the set of semi-norms $|\cdot|_y$ on $R$ whose restriction to $A$ belongs to $\sM(A)$ i.e., whose restriction to $A$ is bounded by $|\cdot|_{\Ban}$. The set $\sM(R)$ is endowed with the coarsest topology making the map $y \mapsto |\cdot|_y$ continuous for $y \in A$, and it comes with a structure map $\sM(R) \to \sM(A)$ sending each $|\cdot|_y$ to its restriction to $A$. 

If $X$ is an $A$-scheme of finite type, then $X$ is covered by open affine schemes $\{ \Spec(R_i)\}$ with $R_i$ being finitely generated $A$-algebras. 
Then, one can glue the analytifications of the affine charts to define the analytic space $X^{\an}$.
We will refer to $X^{\an}$ as a \cdef{Berkovich $A$-analytic space} or as an \cdef{$A$-analytic space}.
The space $X^{\an}$ comes equipped with a structure map $\pi\colon X^{\an} \to \sM(A)$. The kernel maps on each $\sM(R_i)$ induce a continuous kernel map 
\[
\iota\colon X^{\an} \to X.
\] 
Normally, it will be clear which Banach ring we are working over, but sometimes it will be convenient to use the notation $(X/(A,|\cdot|_{\Ban}))^{\an}$ in order to incorporate this data. 

\subsection{Topological properties of Berkovich spaces over commutative Banach rings}
Similar to Berkovich spaces defined over non-Archimedean fields, $X^{\an}$ satisfies nice topological properties. 

\begin{prop}[\protect{\cite[Lemmas 1.1 \& 1.2]{Berkovich:WeightZero}}]\label{prop:Berkovichproperties}
The following properties hold.
\begin{enumerate}
\item If $X$ is of finite type over $A$, then $X^{\an}$ is locally compact and countable at infinity.
\item If $X$ is separated and of finite type over $A$, then $X^{\an}$ is Hausdorff.
\item If $X$ is projective over $A$, then $X^{\an}$ is compact. 
\end{enumerate}
\end{prop}

Using \cite[Definition 1.5.3]{BerkovichSpectral}, any $A$-analytic space comes equipped with a sheaf of analytic functions, and hence the above defines an analytification functor $X \mapsto X^{\an}$ where the category of $A$-analytic spaces is defined in \cite{LemanissierPoineau:BerkZ}. We refer the reader to \cite[Examples 1.5 \& 1.6]{PilleSchneider:Global} for examples of these constructions when when $A$ is $\mC$ and a complete, non-Archimedean field, respectively. 
In these cases, the above constructions topologically recover complex analytic spaces and Berkovich spaces defined over non-Archimedean fields. 

The next proposition tells us that a Berkovich $A$-analytic can be viewed as a family of Berkovich spaces over complete valued fields. 

\begin{prop}[\protect{\cite[Proposition 1.8]{PilleSchneider:Global}}]\label{prop:basechange_homeo}
Let $A$ be a Banach ring, $X$ a finite type $A$-scheme, and $\pi\colon X^{\an} \to \sM(A)$ the associated $A$-analytic space. 
If $x\in \sM(A)$, then $\pi^{-1}(x)$ is canonically homeomorphic to the analytification of the base change $X_{\sH_x} = X \times_A \sH_x$ with respect to the absolute value $|\cdot|_x$ on $\sH_x$. 
Moreover, the base change map $F_x\colon X^{\an}_{\sH_x} \to X^{\an}$ is the inclusion $\pi^{-1}(x) \subset X^{\an}$ under this homeomorphism. 
\end{prop}

\begin{notation}\label{notion:basechange}
Keep the notation as in \autoref{prop:basechange_homeo}. 
For any point $x\in \sM(A)$, we let 
\[
F_x\colon X^{\an}_{\sH_x} \to X^{\an}
\]
denote the base change map. 
\end{notation}

\begin{example}\label{exam:BerkovichoverZ}
Let $(A,|\cdot|_{\Ban}) = (\mZ,|\cdot|)$ where $|\cdot|$ is the standard absolute value, and let $\cX$ be an arithmetic variety. 
The analytification $\pi\colon \cX^{\an} \to \sM(\mZ)$ can be defined via disjoint unions
\[
\cX^{\an} = \bigsqcup_{x\in \sM(\mZ)} \cX_{\sH_x}^{\an}.
\]
Any point $x\in \sM(\mZ)$ is an absolute value of $\mZ$, and from Ostrowski's theorem, these absolute values (and hence points of $\sM(\mZ)$) can be described as follows:
\begin{itemize}
\item the trivial absolute value $x_0 := |\cdot|_0$,
\item the Archimedean absolute values $x_{\infty,\varepsilon} := |\cdot|_{\infty,\varepsilon} = |\cdot|_{\infty}^{\varepsilon}$ with $\varepsilon\in (0,1]$,
\item for any prime $p$, $x_{p,\varepsilon} := |\cdot|_{p,\varepsilon} = |\cdot|_{p}^{\varepsilon}$ with $\varepsilon \in (0,\infty]$ where $|\cdot|_{p,\infty}$ corresponds to the trivial absolute value on $\mF_p$. Note that $|n|_{p,\infty} = 0$ if $p\mid n$ and is 1 otherwise. 
\end{itemize}

As such, $\sM(\mZ)$ can be described as an infinite wedge of segments parametrized by prime numbers $p$ and $\infty$ which are glued together at the trivial absolute value $x_0$. The segments corresponding to the finite primes $p$ are homeomorphic to $[0,\infty]_p$ with the point $0$ corresponding to $x_0$ and $\infty$ to $x_{p,\infty}$ and the segment corresponding to $\infty$ is homeomorphic to $[0,1]_{\infty}$ with similar identifications as before. 
The subscripts on the intervals keep track of the place i.e., this determines the direction of the interval. 
Furthermore, the map $\iota\colon \sM(\mZ)\to \Spec(\mZ)$ sends the points $x_{p,\infty}$ to the prime ideal $(p)$ and every other point to the generic point of $\Spec(\mZ)$. 

Via this description of $\sM(\mZ)$, we can identify the following fibers of $\pi\colon \cX^{\an} \to \sM(\mZ)$:
\begin{enumerate}
\item $\cX_{\sH_{x_0}}^{\an} \cong (\cX_{\mQ}/(\mQ,|\cdot|_0))^{\an}$ i.e, the Berkovich analytification of $\cX_{\mQ}$ over the trivially valued field $(\mQ,|\cdot|_0)$. 
\item $\cX_{\sH_{x_{\infty,1}}}^{\an} \cong (\cX_{\mR}/(\mR,|\cdot|_{\infty,1}))^{\an}$ i.e, the Berkovich analytification of $\cX_{\mR}$ over the Archimedean valued field $(\mR,|\cdot|_{\infty,1})$, which is homeomorphic to $\cX(\mC)/\Gal(\mC/\mR)$.
\item $\cX_{\sH_{x_{p,1}}}^{\an} \cong (\cX_{\mQ_p}/(\mQ_p,|\cdot|_{p,1}))^{\an}$ i.e, the Berkovich analytification of $\cX_{\mQ_p}$ over the non-Archimedean valued field $(\mQ_p,|\cdot|_{p,1})$, and 
\item $\cX_{\sH_{x_{p,\infty}}}^{\an} \cong (\cX_{\mF_p}/(\mF_p,|\cdot|_{p,\infty}))^{\an}$  i.e, the Berkovich analytification of $\cX_{\mF_p}$ over the trivially valued field $(\mF_p,|\cdot|_{p,\infty})$. 
\end{enumerate}
For each prime $p$ and each $\varepsilon \in (0,\infty)$, \cite[Lemme 2.9]{Poineau:Dynamics1} asserts that the spaces $\cX_{\sH_{x_{p,1}}^{\an}}$ and $\cX_{\sH_{x_{p,\varepsilon}}^{\an}}$ are isomorphic as locally ringed spaces. 
A similar statement holds for the fibers over the $(0,1]_{\infty}$. 
\end{example}

\subsection{The reduction map}
Next, we recall the construction of the reduction map for proper Berkovich spaces. 
Let $(A,|\cdot|_{\Ban}) = (\mZ,|\cdot|)$, let $\cX$ be a proper arithemtic variety, and let $\cX^{\an} = (\cX/(\mZ,|\cdot|))^{\an}$ be the analytification with structure morphism $\pi\colon \cX^{\an}\to \sM(\mZ)$. 
Since $\cX$ is proper, there is a reduction map 
\[
\red\colon \cX^{\an} \to \cX
\]
defined via the valuative criterion for properness. 
More precisely, for each $x\in \cX^{\an}\setminus \pi^{-1}((0,1]_{\infty})$ i.e., the non-Archimedean points, the natural map $\Spec (\sH_x) \to \cX$ uniquely extends to a morphism $\Spec (R_{x}) \to \cX$ via the valuative criterion for properness, and we define $\red(x)$ to be the image of the closed point of $\Spec (R_x)$, and for $x \in \pi^{-1}((0,1]_{\infty})$ i.e., the Archimedean points, we define $\red(x)$ to be the image of $\Spec (\sH_x)$ in $\cX$. 
An important feature of $\red$ is that it is anti-continuous meaning that the pre-image of an open subset is closed. 

\subsection{The interior and normalized boundary of Berkovich spaces}
\label{subsec:interior_normalizedboundar}
Later, we will need the notion of the interior and normalized boundary of Berkovich spaces, which was recently introduced by Song \cite{Song:EquivariantAdelic}. 
As evident from \autoref{exam:BerkovichoverZ}, there are many ``redundant'' points of a Berkovich space over $\sM(\mZ)$.  
It is useful to define an equivalence relation on the points. 

\begin{definition}\label{defn:normequivalent}
Let $\cX$ be an arithmetic variety, and let $\cX^{\an} = (\cX/(\mZ,|\cdot|))^{\an}$ be the analytification. 
We say that two points $x,y \in \cX^{\an}$ are \cdef{norm-equivalent}, denoted by $x \sim y$ if $\iota(x) = \iota(y)$ and $|\cdot|_y = |\cdot|_x^t$ for some positive real number $t$ where $\iota$ is the kernel map defined above. We endow $\cX^{\an}/\sim$ with the quotient topology. 
\end{definition}

\begin{definition}\label{defn:interiorboundary}
The \cdef{interior part of $\cX^{\an}$}, denoted by $\cX^{\beth} \subset \cX^{\an}$, is the set of all points in $\cX^{\an}$ such that:
\begin{enumerate}
\item $|a|_x\leq 1$ for all $a\in \mZ$, and 
\item there exists a morphism $\varphi\colon \Spec R_x \to \cX$ such that the commutative diagram
\[
\begin{tikzcd}
\Spec \sH_x \arrow{r} \arrow{d} & \cX \arrow{d} \\
\Spec R_x \arrow{r} \arrow[dashed]{ru}{\varphi} & \Spec \mZ
\end{tikzcd}
\]
commutes. Note that since $\cX$ is separated, $\varphi$ is necessarily unique. 
\end{enumerate} 
The \cdef{boundary part $\cX^b$ of $\cX^{\an}$} is defined to be $\cX^{b} := \cX^{\an} \setminus \cX^{\beth}$, and the \cdef{normalized boundary} is $\wt{\cX}^b := \cX^b/\sim$, again endowed with the quotient topology. 
\end{definition}

\begin{notation}
We can define the reduction map 
\[
\red^{\beth}\colon \cX^{\beth} \to \cX
\]
similar to the proper setting i.e., $\red^{\beth}(x)$ is equal to the image of the closed point of $\Spec R_x$ in $\cX$. 
\end{notation}

%

Let $\cU$ be a quasi-projective arithmetic variety. In \cite{Song:EquivariantAdelic}, Song proved that $\cU^{\beth}$ is a compact, Hausdorff topological space by noting that for any projective model $\cX$ of $\cU$ over $\Spec \mZ$, we have that $\cX^{\beth} \cong (\cX/(\mZ,|\cdot|_0))^{\an}$, which is compact, Hausdorff by \autoref{prop:Berkovichproperties}.(3). Moreover, we can realized $\cU^{\beth} = \red^{\beth, -1}(\cU)$ where $\red^{\beth}\colon \cX^{\beth} \to \cX$, and the result follows from the anti-continuity of $\red^{\beth}$. 
Additionally, he proved that $\wt{\cU}^b$ is compact, Hausdorff using the fact that
\[
\sM(\mZ,|\cdot|) \cong \sM(\mZ,|\cdot|_{0}) \cup \sM(\mQ,|\cdot|_{\text{hyb}})
\]
where $(\mZ,|\cdot|_0)$ (resp.~$(\mQ,|\cdot|_{\text{hyb}})$) is $\mZ$ (resp.~$\mQ$) endowed with the trivial norm (resp.~the hybrid norm). 
We refer the reader to \cite[Section 2.2]{Song:EquivariantAdelic} for details on hybrid norms, but one should think of $\sM(\mQ,|\cdot|_{\text{hyb}})$ as corresponding to the Archimedean segment of $\sM(\mZ)$. 

In summary, we have the following result. 

\begin{theorem}[\protect{\cite[Theorem 2.5]{Song:EquivariantAdelic}}]
Let $\cU$ be a quasi-projective arithmetic variety, and let $\cU^{\an} = (\cU/(\mZ,|\cdot|))^{\an}$ be the analytification. 
Then, 
\begin{enumerate}
\item $\cU^{\beth}$ is a compact, Hausdorff subspace of $\cU^{\an}$, and 
\item $\wt{\cU}^b$ is a compact, Hausdorff topological space when endowed with quotient topology. 
\end{enumerate}
\end{theorem}

\section{Global pluripotential theory on Berkovich spaces}
\label{sec:globalpluripotential}
In this section, we recall global pluripotential theory on Berkovich spaces. 
We refer the reader to \cite{BoucksomEriksson:SpaceNorms, BoucksomJonsson:SingularPSH, PilleSchneider:Global} for details on metrized line bundles on Berkovich spaces and pluripotential theory over $\mC$, non-Archimedean non-trivially valued fields, and trivially valued fields. 

Let $(A,|\cdot|_{\Ban})$ be an integral Banach ring, let $X$ be an $A$-variety, let $L$ be a line bundle on $X$, and let $X^{\an}$ denote the Berkovich $A$-analytic space associated to $X$.

\subsection{Metrized line bundles on Berkovich spaces}
\label{subsec:metrizedlinebundles}
First, we describe the general setup of metrics on line bundles on Berkovich spaces. 
At each point $x\in X^{\an}$, let $\iota(x)$ denote the image of $x$ in $X$.

\begin{definition}
The \cdef{fiber} $L^{\an}_x$ of $L^{\an}$ at $x$ is defined to be $\sH_x$-line $L_{\iota(x)} \otimes_{\kappa(\iota(x))} \sH_x$ or equivalently the completion of the fiber $L_{\iota(x)}$ of $L$ at $\iota(x)$ with respect to the semi-norm $|\cdot|_x$. 
Let $L^{\an,\times}$ denote the complement of $L^{\an}$ of the image of the zero section, and let $L^{\an,\times}_x \coloneqq L^{\an,\times} \cap L^{\an}_x$. 
\end{definition}

\begin{definition}
A \cdef{metric} of $L$ over $X^{\an}$ consists of the following data:~for each $x\in X^{\an}$, there is a norm $\nrm{\cdot}_x\colon L^{\an}_x \to \mR_{\geq 0}$ defined via $\nrm{v}_x = c(x)|T(v)|_x$ where $c(x)\in (0,\infty)$, which is compatible with the norm $|\cdot|_x$ on $\sH_x$ in the sense that for any $a\in \sH_x$, $v\in L^{\an,\times}_x$, $\nrm{av}_x = |a|_{x}\nrm{v}_x$. 
We will also use the additive notion for metrics by identifying $\nrm{\cdot}$ with $\phi := -\log\nrm{\cdot}$, and let $\phi_x := -\log\nrm{\cdot}_x$. 
The pair $(L,\nrm{\cdot})$ (or equivalently $(L,\phi)$) is called a \cdef{metrized line bundle} over $X^{\an}$. 
\end{definition}

We define several properties attached to metrized line bundles. 

\begin{definition}
Let $(L,\phi)$ be a metrized line bundle over $X^{\an}$. 
\begin{enumerate}
\item The metric $\phi$ on $L$ is \cdef{continuous} if for any section $s$ on a Zariski open subset $U$ of $X$, the function $\phi(s(x)) = \phi(s(x))_x$ is continuous in $x\in U^{\an}$. 
\item The metric $\phi$ on $L$ is \cdef{norm-equivariant} if for any rational section $s$ of $L$ and any points $x, x_1 \in X^{\an}\setminus |\ddiv(s)|^{\an}$ satisfying $|\cdot|_x = |\cdot|_{x_1}^t$ for some $0\leq t < \infty$ locally over $\cO_{X^{\an}}$, we have that $\phi(s(x))_x = t\cdot \phi(s(x_1))_{x_1}$. 
\end{enumerate}
\end{definition}

Next, we define singular metrics.

\begin{definition}
A \cdef{singular metric} on a line bundle $L$ over $X^{\an}$ is a function $\phi\colon L \to \mR \cup \{-\infty\}$ such that for every $x\in X^{\an}$, the restriction $\phi_{L_x^{\an}}$ is either $\equiv -\infty$ or a metric on $L_x^{\an}$. 
\end{definition}

The main example of singular metrics comes from a global section of a line bundle. 

\begin{example}\label{example:sectionmetric}
Let $m\geq 1$ and let $s$ be any global section of $H^0(X,mL)$. 
The section $s$ defines a singular metric
\[
\phi \coloneqq \log |s|
\]
on $L$ as follows:~for any $x\in X^{\an}$, we have that $\phi(s(x)) = 0$ if $s(x)\neq 0$ and $\phi \equiv -\infty$ on $L_x^{\an}$ if $s(x) = 0$. 
In other words, 
\[
\phi(v)_x = \log |v/s(x)|_x
\]
for $v\in L_x^{\an}$. 
Note that $\phi$ is singular on the zero locus of $s$. 
\end{example}

\subsection{Base change of metrics}
\label{subsec:basechange}
Before we discuss pluripotential theory, we describe how metrics behave under base change of Banach rings. 
Let $(A,|\cdot|_A)$ and $(B,|\cdot|_B)$ be two Banach rings with a bounded ring homomorphism $A\to B$. This induces a continuous map $F_B\colon X_B^{\an} \to X^{\an}$ via the fibered diagram
\[
\begin{tikzcd}
X_B^{\an} \arrow{r}{F_B} \arrow{d} & X^{\an}\arrow{d} \\
\sM(B) \arrow{r} & \sM(A).
\end{tikzcd}
\]
For a line bundle $L$ on $X$, there is an induced line bundle $L_B^{\an} \coloneqq L^{\an} \otimes_{\cO_{X^{\an}}} \cO_{X_B^{\an}}$, which can be identified with $F_B^*L^{\an}$. This map induces a map $F_B^*L^{\an} \cong L_B^{\an} \to L^{\an}$, and if $\phi$ is a metric on $L^{\an}$, then we have an induced metric $F_B^*\phi \coloneqq \phi \circ F_B$ on $L_{B}^{\an}$. 
We refer the reader to \cite[Paragraph after Definition 2.7]{PilleSchneider:Global} for further details.

\begin{example}\label{example:base_change_metrics}
Let $(A,|\cdot|_{\Ban})$ be a Banach ring, and let $x\in \sM(A)$. By definition, we have a bounded homomorphism of Banach rings $A \to \sH_x$, and so any metric $\phi$ on $L^{\an}$ induces by base change a metric $\phi_x$ on $(X_{\sH_x}^{\an},L_{x}^{\an})$. 
By \autoref{prop:basechange_homeo}, the metric can also be identified as the restriction of $\phi$ to the fiber $\pi^{-1}(x)$ of the structure map $\pi\colon X^{\an} \to \sM(A)$. 
\end{example}

\subsection{Continuous plurisubharmonic metrics}
We now turn the main objects of pluripotential theory on general Berkovich spaces. 
In order to keep our discussion concise, we will not recall this theory over fields, but instead refer the reader to \cite{BoucksomJonsson:SingularPSH, BoucksomEriksson:SpaceNorms}. 
The theory we present was developed by Pille-Schneider \cite{PilleSchneider:Global} and was motivated by the complex and non-Archimedean theories. 

Let $L$ be a \cdef{semiample} line bundle on $X$, and recall that $L$ being semiample asserts that there is some multiple of it which is base point free. 

First, we define the building blocks of plurisubharmonic metrics. 

\begin{definition}\label{defn:tropicalFS}
A \cdef{tropical Fubini--Study metric} on $L^{\an}$ is a continuous metric of the form
\[
\phi \coloneqq m^{-1} \max_{j\in J}(\log |s_j| + \lambda_j)
\]
where $(s_j)_{j\in J}$ is a finite collection of sections of $mL$ without common zeros, $\lambda_j \in \mR$, and $\log |s_j|$ refers to the metric defined in \autoref{example:sectionmetric}.

We write $\FS^{\tau}(L)$ for the set of tropical Fubini--Study metrics on $L$, and if $L = \cO_{X^{\an}}$, then we simply say that $\phi$ is a tropical Fubini--Study function on $X$. 
If the constants $\lambda_j$ are all zero above, then we say that $\phi$ is a \cdef{pure} tropical Fubini--Study metric.
\end{definition}

We refer the reader to \cite[Proposition 2.14]{PilleSchneider:Global} for basic properties of tropical Fubini--Study metrics. 
Before defining plurisubharmonic metrics, we need a definition.

\begin{definition}\label{defn:Zariksidensepoints}
Let $(A,|\cdot|_{\Ban})$ be an integral Banach ring, and write $\eta_A$ for the generic point of $\Spec(A)$. 
The subset of \cdef{Zariski dense} points $\sM(A)^{\eta} \subset \sM(A)$ is defined to be $\iota^{-1}(\eta_A)$ where $\iota$ is the kernel map from \eqref{eqn:kernelmap}. 
\end{definition}

\begin{example}
From \autoref{exam:BerkovichoverZ}, we see that the only non-Zariski dense points of $\sM(\mZ)$ are $\brk{x_{p,\infty} : p \text{ prime}}$. 
\end{example}

We now arrive a Pille-Schneider's definition of plurisubharmonic metrics.

\begin{definition}[\protect{\cite[Definition 2.18]{PilleSchneider:Global}}]\label{defn:psh}
A \cdef{plurisubharmonic} metric $\phi$ on $L^{\an}$ is a singular metric on $L^{\an}$ that is the pointwise limit of a decreasing net of tropical Fubini--Study metrics on $L^{\an}$ and such that $\phi_x \not\equiv -\infty$ for all $x\in \sM(A)^{\eta}$ where $\phi_x$ is the restriction of $\phi$ to $X^{\an}_{\sH_x}$. 
We write $\PSH(L,X)$ or $\PSH(L)$ for the set of plurisubharmonic metrics on $L$. 
A metric is said to be \cdef{continuous plurisubharmonic} if it is continuous and plurisubharmonic as defined above. 
\end{definition}

Taking motivation from Yuan--Zhang \cite{YuanZhang:AdelicLineBundles}, we introduce a new notion of continuous semipositive metrics on analytificaitons line bundles over quasi-projective $A$-varieties. 
First, we need a definition.  

\begin{definition}\label{defn:projectivemodel}
Let $U$ be a quasi-projective $A$-variety, and let $L$ be a line bundle on $U$. 
We say that the pair $(X,L')$ is a \cdef{projective $A$-model for $(U,L)$} if there exists an open $A$-immersion $\iota\colon U \hookrightarrow X$ (i.e., $X$ is a projective $A$-model for $U$) and $L' \in \Pic(X)_{\mQ}$ such that $\iota^*L' \cong L$. 
\end{definition}

\begin{definition}\label{defn:continuous_semipositive}
Let $U$ be a quasi-projective $A$-variety. 
We say a metric $\phi$ on $L^{\an}$ over $U^{\an}$ is \cdef{continuous semipositive} if there exists the following data:
\begin{itemize}
\item a net of projective $A$-models $(X_i,L_i)$ of $(U,L) $ such that
\item each $L_i^{\an}$ admits a tropical Fubini--Study metric $\phi_i$ (\autoref{defn:tropicalFS}) where
\item the restrictions $\phi_{i|U^{\an}}$ compactly converge to $\phi$ on $U^{\an}$  i.e., for every compact subset $K\subset U^{\an}$, we have that $\phi_{i|K}$ converges uniformly to $\phi_{|K}$. 
\end{itemize}
Note that $\phi_{i|K}$ converges uniformly to $\phi_{|K}$ if and only if $\phi_{i|K} - \phi_{|K}$ converges uniformly to the zero function on $K$. 
Also, since $U^{\an}$ is locally compact (\autoref{prop:Berkovichproperties}.(1)), we have that $\phi_{i|U^{\an}}$ converges pointwise to $\phi$, and the compact convergence of $\phi_{i|U^{\an}}$ implies that $\phi$ is in fact a continuous metric, as implied by the name. 
\end{definition}

We illustrate how this relates to previous notions.

\begin{lemma}\label{lemma:proj_continuoussemipositive_continuouspsh}
Let $X$ be a projective $A$-variety, and $L$ a line bundle on $X$. 
A metric $\phi$ is continuous semipositive if and only if $\phi$ is continuous plurisubharmonic in the sense of \cite{PilleSchneider:Global} (\autoref{defn:psh}). 
\end{lemma}

\begin{proof}
When $X$ is a projective $A$-scheme with line bundle $L$, the only projective $A$-model for $(X,L)$ is $(X,L)$. 
Indeed, any $A$-morphism from between projective $A$-varieties will be projective, and hence will have closed image. 
Since we require an $A$-variety to be integral, this implies the statement.

Additionally, we note that $X^{\an}$ will be a compact Berkovich $A$-analytic space by \autoref{prop:Berkovichproperties}.(3). 
As such, we see that a metric $\phi$ is continuous plurisubharmonic if and only if there exists a net of tropical Fubini--Study metrics converging uniformly to $\phi$. 
To conclude, we note that \cite[Proposition 2.21]{PilleSchneider:Global} asserts that this is equivalent to being continuous plurisubharmonic.
\end{proof}

\begin{lemma}\label{lemma:proj_restrictedcontinuoussemipositive}
Let $X$ be a projective $\mZ$-scheme, let $\overline{L} = (L,\phi)$ be a metrized semiample line bundle with $\phi$ being continuous semipositive. 
Then, for every $x\in \sM(\mZ)$, the restriction of $\phi$ to the fiber $\pi^{-1}(x)\cong X^{\an}_{\sH_x}$ is a continuous semipositive metric in the appropriate sense. 
\end{lemma}

\begin{proof}
Note that for every $x\in \sM(\mZ)$, $X^{\an}_{\sH_x}$ is a projective Berkovich $\sH_x$-analytic space.
In the proof of \autoref{lemma:proj_continuoussemipositive_continuouspsh}, we showed there exists a net of tropical Fubini--Study metrics converging uniformly to $\phi$. 
As each $X^{\an}_{\sH_x}$ is compact, we have that the restriction of $
\phi$ to $X^{\an}_{\sH_x}$, denoted by $\phi_x$, can be represented as the uniform limit of a net of tropical Fubini--Study metrics as tropical Fubini--Study metrics pullback to tropical Fubini--Study metrics by \cite[Proposition 2.14.(7)]{PilleSchneider:Global}. 
The claim now follows from \cite[Proposition 5.20]{BoucksomJonsson:SingularPSH} for the non-Archimedean points and \autoref{thm:eqiuvalentcomplexHermitianline} for the Archimedean points. 
\end{proof}
\begin{notation}
We introduce the following notation for the groups and categories we have just introduced. 
As above, let $(A,|\cdot|_{\Ban})$ denote an integral Banach ring, and let $X$ be an $A$-variety. 
We define the following categories
\begin{align*}
\acPic(X^{\an})_{\eqv} &\coloneqq \text{category of continuous, norm-equivariant metrized line bundles on $X^{\an}$},\\
\acPic(X^{\an})_{\eqv,\tFS} &\coloneqq \text{category of norm-equivariant, tropical Fubini--Study metrized line bundles on $X^{\an}$},\\
\acPic(X^{\an})_{\eqv,\semip} &\coloneqq \text{category of norm-equivariant, continuous semipositive metrized line bundles on $X^{\an}$}.
\end{align*}
where the objects in these categories are the corresponding metrized line bundles, and the morphisms are isometries, namely metric preserving isomorphism. 
We will also use this notation without the underline to simply denote the corresponding group of such line bundles. 
\end{notation}
\subsection{Arithmetic divisors on Berkovich spaces}
We now define arithmetic divisors on Berkovich spaces. 
Recall that $(A,|\cdot|_{\Ban})$ is an integral Banach ring, $X$ is an $A$-variety, and $X^{\an}$ is the Berkovich $A$-analytic space associated to $X$. 

\begin{definition}
Let $D$ be a Cartier divisor on $X$.
A \cdef{Green's function} of $D$ over $X^{\an}$ is a continuous function $g_D\colon X^{\an}\setminus |D|^{\an} \to \mR$ with logarithmic singularity along $D$ in the sense that for any rational function $f$ over a Zariski open subset $U$ of $X$ such that $\ddiv(f) = D_{|U}$, the function $g_D + \log |f|$ can be extended to a continuous function on $U^{\an}$. 
The pair $\overline{D} = (D,g_D)$ is called an \cdef{arithmetic divisor} over $X^{\an}$. 

We say that an arithmetic divisor is \cdef{effective} if $D$ is effective and $g_D$ is non-negative over $X^{\an}\setminus |D|^{\an}$. 
An arithmetic divisor is called \cdef{principal} if it is of the form
\[
\widehat{\ddiv}_{X^{\an}}(f)\coloneqq (\ddiv(f),-\log |f|)
\]
for some non-zero rational function $f$ over $X$. 
We stress that the divisor $D$ in $\overline{D}$ is algebraic, and the rational function in a principal arithmetic divisor is also algebraic. 

An arithmetic divisor $\overline{D} = (D,g_{D})$ or its Green's function $g_{D}$ is called \cdef{norm-equivariant} if for any points $x_1,x_2\in X^{\an}\setminus |D|^{\an}$ satisfying $|\cdot|_{x} = |\cdot|_{x_1}^t$ for some $0\leq t < \infty$ locally over $\cO_{X^{\an}}$, we have that $g(x) = tg(x_1)$. 
Note that principal arithmetic divisors are norm-equivariant. 
\end{definition}

The pair $(D^{\an},g_D)$ defines a metrized line bundle $(\cO(D),\nrm{\cdot}_g)$ with the metric defined by $\nrm{s_D}_g = e^{-g}$ where $s_D$ is a section of $\cO(D^{\an})$ corresponding to the meromorphic function $1$ on $X^{\an}.$
By this correspondence, we have that $g$ is continuous if and only if $\nrm{\cdot}_g$ is continuous. 
We make the following definition.

\begin{definition}\label{defn:Berk_tfs_semi_arithmeticdivisors}
For $D$ an Cartier divisor on $X$ with Green's function $g_D$, we say that $g_D$ is \cdef{tropical Fubini--Study} (resp.~\cdef{continuous semipositive}) if $\nrm{\cdot}_g$ is a tropical Fubini--Study (resp.~continuous semipositive) metric. 
\end{definition}

\begin{notation}
With the above definition, we have the following groups.
\begin{align*}
\aDiv(X^{\an}) &\coloneqq \text{ group of arithmetic divisors on $X^{\an}$},\\
\aDiv(X^{\an})_{\eqv} &\coloneqq \text{ group of norm-equivariant arithmetic divisors on $X^{\an}$},\\
\aDiv(X^{\an})_{\eqv, \tFS} &\coloneqq \text{ group of norm-equivariant, tropical  Fubini--Study arithmetic divisors on $X^{\an}$},\\
\aDiv(X^{\an})_{\eqv, \semip} &\coloneqq \text{ group of norm-equivariant, continuous semipositive arithmetic divisors on $X^{\an}$},\\
\aPr(X^{\an}) &\coloneqq \text{ group of principal arithmetic divisors on $X^{\an}$}.
\end{align*}
As in the classical theory, we can form the quotient of the first two groups by the last one to realize the class group of arithmetic divisors; in particular, we introduce the following groups:
\begin{align*}
\aCaCl(X^{\an}) &\coloneqq \aDiv(X^{\an})/\aPr(X^{\an}) \hspace{.5em}(\text{arithmetic class group of $X^{\an}$}),\\
\aCaCl(X^{\an})_{\eqv} &\coloneqq \aDiv(X^{\an})_{\eqv}/\aPr(X^{\an}) \hspace{.5em}(\text{norm-equivariant arithmetic class group of $X^{\an}$}).
\end{align*}
\end{notation}
As in the classical setting, we have an isomorphism between these class groups and the Picard groups introduced in Subsection \ref{subsec:metrizedlinebundles}. 

\begin{prop}[\protect{\cite[Section 3.2.2]{YuanZhang:AdelicLineBundles}}]\label{prop:classisom_Berkovich}
There are canonical isomorphisms
\begin{align*}
\aCaCl(X^{\an}) & \cong \aPic(X^{\an}) \\
\aCaCl(X^{\an})_{\eqv} & \cong \aPic(X^{\an})_{\eqv}.
\end{align*}
\end{prop}

This leads to the following definition.

\begin{definition}\label{defn:classgroup_eqvsemip_Berk}
We define the \cdef{class group of norm-equivariant, continuous semipositive arithmetric divisors} on $X^{\an}$, denoted by $\aCaCl(X^{\an})_{\eqv, \semip}$, to be the inverse image of $\aPic(X^{\an})_{\eqv,\semip}$ under the isomorphism from \autoref{prop:classisom_Berkovich}. 
\end{definition}

\section{Metrized line bundles and arithmetic divisors}
\label{sec:metrizedprojectivearithmetic}
In this section, we recall the definition of metrized line bundles and arithmetic divisors on projective arithmetic varieties and conclude by discussing how to analytify these objects. 
We remind the reader of our conventions from Subsection \ref{subsec:conventions_AG} regarding projective arithmetic varieties. 

\subsection{Hermitian line bundles on projective arithmetic varieties}
To begin our discussion on metrized line bundles, we first recall the complex analytic background.

\subsubsection{Metrized line bundles on complex analytic spaces}
Let $X(\mC)$ be a reduced and irreducible complex analytic variety endowed with the complex analytic topology. Let $L$ be a line bundle on $X(\mC)$.

\begin{definition}
A \cdef{continuous metric} of $L$ on $X(\mC)$ is the data of a metric $\nrm{\cdot}$ on the fiber $L_x$ above every point $x\in X(\mC)$ which varies continuously in that for any local section $s$ of $L$ defined on an open subset $U$ of $X$, the function $\nrm{s(x)}$ is continuous in $x\in U$. We will often use the additive notation for metrics by setting $\phi := -\log \nrm{\cdot}$. 
\end{definition}

For any continuous metric $\nrm{\cdot}$ of $L$ on $X(\mC)$, the \cdef{Chern current} 
\[
c_1(L,\nrm{\cdot}) \coloneqq \frac{1}{\pi i}\partial \overline{\partial} \log\nrm{s} + \delta_{\ddiv(s)}
\]
is a $(1,1)$-current on $X(\mC)$, where $s$ is any meromorphic section of $L$ on $X(\mC)$.  

\begin{definition}
We say that a continuous metric $\nrm{\cdot}$ of $L$ on $X(\mC)$ is called \cdef{semipositive} if the Chern current is a positive current or equivalently if for any local section $s$ of $L$ which is analytic and everywhere non-vanishing on an open subset $U$ of $X(\mC)$, the function $-\log\nrm{s(x)}$ is plurisubharmonic on $U$. 
\end{definition}

An important result in complex pluripotential theory is Demailly's regularization theorem \cite{Demailly:Regularization} which illustrates how semipositive metrics can be described from simpler metrics when $X(\mC)$ is smooth and $L$ is ample. 
There are variants of this result which apply to the non-smooth setting and when $L$ is semiample due to \cite[Theorem 7.1]{BoucksomEriksson:SpaceNorms}. 
For our purposes, we will need a version of this result from \cite{PilleSchneider:Global}. 
First, we recall a definition.

\begin{definition}\label{defn:Hermitian_TFS}
A continuous metric $\phi = -\log\nrm{\cdot}$ of $L$ on $X(\mC)$ is said to be a \cdef{tropical Fubini--Study metric} if 
\[
\phi  = m^{-1}\max_{j\in J} (\log |s_j| + \lambda_j)
\]
where $(s_j)_{j\in J}$ is a finite collection of global sections of $mL$ without common zero and $\lambda_j \in \mR$. 
\end{definition}

Clearly, $L$ admits a tropical Fubini--Study metric if and only if $L$ is semiample. 
Next, we recall the important result that semipositive metrics can be described via uniform limits of tropical Fubini--Study metrics. 
We note that there is a more general statement for singular plurisubharmonic metrics, which we will not need. 

\begin{theorem}\label{thm:eqiuvalentcomplexHermitianline}
Let $X(\mC)$ be a projective complex variety, let $L$ be a semiample line bundle on $X$, and let $\phi$ be a continuous metric on $L$. 
Then, the following are equivalent:
\begin{enumerate}
\item $\phi$ is a continuous semipositive metric,
\item $\phi$ is a decreasing limit of a net of tropical Fubini--Study metrics,
\item $\phi$ is a uniform limit of a net of tropical Fubini--Study metrics. 
\end{enumerate}
\end{theorem}

\begin{proof}
The equivalence of the first and second statement follows from \cite[Theorem 2.11]{PilleSchneider:Global}. 
That $(2) \Rightarrow (3)$ follows from Dini's theorem, and the reverse implication is elementary.
\end{proof}

For the remainder of the section, let $\cX$ be a projective arithmetic variety. 

\begin{definition}
A \cdef{Hermitian line bundle} on $\cX$ is a pair $\overline{\cL} = (\cL,\nrm{\cdot})$ where $\cL$ is a line bundle on $\cX$ and $\nrm{\cdot}$ is a continuous metric of $\cL(\mC)$ on $\cX(\mC)$, which is invariant under the action of complex conjugation. 
\end{definition}

In order to establish an equivalence between certain subcategories of adelic line bundles and metrized line bundles on Berkovich analytic spaces, we will need to specify the type of metric on $\cL(\mC)$.  
We do so by imposing a global constraint on the line bundle which induces an Archimedean condition on the metric.

\begin{definition}\label{defn:GTFSH}
A \cdef{global tropical Fubini--Study} Hermitian line bundle on $\cX$ is a Hermitian line bundle $\overline{\cL} = (\cL,\phi)$ where $\cL$ is a semiample line bundle on $\cX$ and 
\[
\phi = m^{-1}\max_{j\in J} (\log |s_{j}\otimes 1_{\mC}| + \lambda_{j})
\]
where $(s_{j})_{j\in J}$ are global sections of $H^0(\cX,m\cL)$ which do not have common zeros, $\lambda_{j} \in \mR$, and $\phi$ is invariant under the action of complex conjugation. 
If the constants $\lambda_j$ are zero, we say that a global tropical Fubini--Study Hermitian line bundle is \cdef{pure}.  
We call the associated metric $\phi$ on $\cL(\mC)$ a \cdef{global tropical Fubini--Study metric}.

A \cdef{global semipositive} Hermitian line bundle on $\cX$ is a Hermitian line bundle $\overline{\cL} = (\cL,\phi)$ where $\cL$ is a semiample line bundle on $\cX$ and $\phi$ is a semipositive metric of $\cL(\mC)$ on $\cX(\mC)$ which can be written as the limit of a decreasing net of global tropical Fubini--Study metrics and is invariant under the action of complex conjugation. 
\end{definition}

\begin{remark}
The difference between a tropical Fubini--Study (\autoref{defn:Hermitian_TFS}) and global tropical Fubini--Study Hermitian line bundle (\autoref{defn:GTFSH}) lies in the conditions concerning the set of common zeros of the sections defining them.  
By definitions, it is clear that a global tropical Fubini--Study Hermitian line bundle is a tropical Fubini--Study Hermitian line bundle, but the converse need not be true.  
The phrase ``global'' in our terminology is meant to reflect the condition that the metric at the Archimedean place is induced by a global condition on the line bundle over the arithmetic variety. 
\end{remark}

%

\begin{notation}
As above, let $\cX$ be a projective arithmetic variety, and let $\cL$ be a line bundle on $\cX$. 
We introduce the following notation for the groups and categories we have just introduced. 
Namely, let  
\begin{align*}
\acPic(\cX) &\coloneqq \text{ category of Hermitian line bundles on $\cX$}\\
\acPic(\cX)_{\gtFS} &\coloneqq \text{ category of global tropical Fubini--Study Hermitian line bundles on $\cX$} \\
\acPic(\cX)_{\gsemip} &\coloneqq \text{ category of global semipositive Hermitian line bundles on $\cX$}. 
\end{align*}
where the objects are the corresponding Hermitian line bundles, and every morphism is an isometry, namely a metric preserving isomorphism. 
We will also use this notation without the underline to simply denote the corresponding group of such line bundles. 
\end{notation}

\subsection{Arithmetic divisors on projective arithmetic varieties}
\label{subsec:arithmeticdivisors}
Next, we recall arithmetic divisors on projective arithmetic varieties. 

\begin{definition}
Let $D(\mC)$ be an analytic Cartier divisor on $X(\mC)$ with support $|D|$. 
A \cdef{continuous Green's function} $g$ of $D(\mC)$ on $X(\mC)$ is a function $g\colon X(\mC)\setminus |D| \to \mR$ such that for any meromorphic function $f$ on an open subset $U$ of $X(\mC)$ satisfying $\ddiv(f) = D_{|U}$, the function $g + \log |f|$ can be extended to a continuous function on $U$. 
\end{definition}

The pair $(D(\mC),g)$ defines a metrized line bundle $(\cO(D),\nrm{\cdot}_g)$ with the metric defined by $\nrm{s_D}_g = e^{-g}$ where $s_D$ is a section of $\cO(D)$ corresponding to the meromorphic function $1$ on $X(\mC).$
By this correspondence, we have that $g$ is continuous if and only if $\nrm{\cdot}_g$ is continuous. 
We can also make the following definition using this relationship. 

\begin{definition}
For $D(\mC)$ an analytic Cartier divisor on $X(\mC)$ with Green's function $g$, we say that $g$ is \cdef{tropical Fubini--Study} (resp.~\cdef{semipositive}) if $\nrm{\cdot}_g$ is a tropical Fubini--Study (resp.~semipositive) metric. 
\end{definition}

Using this, we can define arithmetic divisors on projective arithmetic varieties. 

\begin{definition}\label{defn:arithmeticdivisor}
Let $\cX$ be a projective arithmetic variety. 
An \cdef{arithmetic divisor} on $\cX$ is a pair $\overline{\cD} = (\cD,g_{\cD})$ where $\cD$ is a Cartier divisor on $\cX$ and $g_{\cD}$ is a Green's function on $\cD(\mC)$ i.e., $g_{\cD}$ is a continuous, real-valued, conjugate-invariant function on $(\cX\setminus \cD)(\mC)$ with respect to the complex analytic topology such that if $\cD$ is locally defined by $f$, then $g_{\cD} + \log |f|$ extends to a continuous function on $\cD(\mC)$. 
An arithmetic divisor $(\cD,g_{\cD})$ is called \cdef{effective} (resp.~\cdef{strictly effective}) if $\cD$ is effective and $g_{\cD}\geq 0$ (resp.~$g_{\cD}>0$).

A \cdef{principal arithmetic divisor} on $\cX$ is an arithmetic divisor of the form
\[
\widehat{\ddiv}(f) \coloneqq (\ddiv(f),-\log |f|)
\]
for any rational function $f\in \mQ(\cX)^{\times}$ on $\cX$. 
\end{definition}

As in the complex analytic setting, we can define certain classes of arithmetic divisors. 
Note that the pair $(\cD(\mC),g_{\cD})$ defines a Hermitian line bundle $(\cO(\cD),\nrm{\cdot}_{g_{\cD}})$ with the metric defined by $\nrm{s_{\cD}}_{g_{\cD}} = e^{-{g_{\cD}}}$ where $s_{\cD}$ is a section of $\cO(\cD)$ corresponding to the meromorphic function $1$ on $\cX.$
By this correspondence, we have that $g_{\cD}$ is continuous if and only if $\nrm{\cdot}_{g_{\cD}}$ is continuous. 
Similar to the above, we can also make the following definition using this relationship. 

\begin{definition}
For $\cD$ a Cartier divisor on $\cX$ with Green's function $g_{\cD}$, we say that $\overline{D} = (\cD,g_{\cD})$ or just $g_{\cD}$ is \cdef{global tropical Fubini--Study} if $\nrm{\cdot}_g$ is a global tropical Fubini--Study metric. 
Similarly, we say that $g_{\cD}$ is \cdef{global semipositive} if $\nrm{\cdot}_g$ is a global semipositive metric. 
We will adopt similar naming conventions for arithmetic divisors $(\cD,g_{\cD})$ when the Green's function $g_{\cD}$ possess those properties. 
\end{definition}

\begin{notation}
We introduce the following notation for the groups we have just introduced. 
As above, let $\cX$ be a projective arithmetic variety.
\begin{align*}
\aDiv(\cX) &\coloneqq \text{group of arithmetic divisors on $\cX$},\\
\aPr(\cX) &\coloneqq \text{group of principal arithmetic divisors on $\cX$},\\
\aDiv(\cX)_{\gtFS} &\coloneqq \text{group of global tropical Fubini--Study arithmetic divisors on $\cX$},\\
\aDiv(\cX)_{\gsemip} &\coloneqq \text{group of global semipositive arithmetic divisors on $\cX$}.
\end{align*}

Using these notation, we can also form the corresponding \cdef{arithmetic divisor class group of $\cX$} via the quotient
\begin{align*}
\aCaCl(\cX) &\coloneqq \aDiv(\cX)/\aPr(\cX).  
\end{align*}
\end{notation}
Recall that as in the classical setting, we have the isomorphism
\[
\aCaCl(\cX)\cong \aPic(\cX)
\]
given by sending $\overline{\cD} = (\cD,g_{\cD})$ to $(\cO(\cD),\nrm{\cdot}_{g_{\cD}})$ and the inverse image of a Hermitian line bundle $\overline{\cL} = (\cL,\nrm{\cdot})$ is represented by 
\[
\widehat{\ddiv}(s) \coloneqq (\ddiv(s),-\log\nrm{s}). 
\]

\subsection{Analytification of Hermitian line bundles}
To conclude, we will recall the analytification map of Yuan and Zhang from $\aPic(\cX)$ to $\aPic(\cX^{\an})$ when $\cX$ is a projective arithmetic variety. 
Additionally, we will introduce a new analytification map and illustrate when it coincides with Yuan and Zhang's.

\subsubsection{The Yuan--Zhang analytification in projective setting}
Let $\cX$ be a projective arithmetic variety. 
Yuan and Zhang define a canonical map
\[
\aPic(\cX) \to \aPic(\cX^{\an})_{\eqv}
\]
using model metrics. 
More precisely, let $\overline{\cL} = (\cL,\nrm{\cdot})$ denote a Hermitian line bundle. 
We need to define a metric of $\cL^{\an}$ over $\cX^{\an}$.

First, we can define a metric on the fiber of $\cL^{\an}$ over $\cX_{\sH_{x_{\infty,1}}}^{\an}$ using the original Hermitian metric (as it is invariant under complex conjugation), and then we extend it to the fiber $\cL^{\an}$ over $\cX_{\sH_{x_{\infty,i}}}^{\an}$ for $i \in (0,1]_{\infty}$ by norm-equivalence. 
Over the non-Archimedean points $x\in \cX^{\an}$, let $\phi_x^{\circ}\colon \Spec(R_x) \to \cX$ denote the $\mZ$-morphism extending the $\mZ$-morphism $\phi_x\colon \Spec(\sH_x) \to \cX$ by valuative criterion of properness. 
Then, $(\phi_x^{\circ})^*\cL$ is a free module over $R_x$ of rank 1, and if $s_x$ is a basis for this free module, then Yuan--Zhang declare that the metric $
\nrm{\cdot}_x$ of $\cL_x = \phi_x^{*}\cL$ is defined by setting $\nrm{s_x}_x = 1$. In other words, over each non-Archimedean point, Yuan and Zhang define the metric on $\cL_x$ by considering the \cdef{model metric} induced by $(\cX_{R_x},(\phi_x^{\circ})^*\cL)$. We refer the reader to \cite[Section 2.7]{BoucksomJonsson:SingularPSH} for more details on model metrics. 

To establish notation, denote the above analytification of $\overline{\cL} = (\cL,\nrm{\cdot})$ by $\overline{\cL}^{\an} \coloneqq (\cL^{\an},\nrm{\cdot}^{\an}_{\YZ})$. 
We will also adopt this notation in the additive setting. 

\subsubsection{Another analytification map for global tropical Fubini--Study metrics}
\label{subsubsec:analytificationnew}
We now introduce another analytification map for Hermitian line bundles equipped with a global tropical Fubini--Study metric. 
In other words, we will define a map
\[
\aPic(\cX)_{\gtFS} \to \aPic(\cX^{\an})_{\eqv, \tFS}. 
\]

Let $\overline{\cL} = (\cL,\phi)$ denote a global tropical Fubini--Study Hermitian line bundle; recall that this means that $\phi$ is metric on $\cL(\mC)$ of the form
\[
\phi = m^{-1}\max_{j\in J} (\log |s_{j}\otimes 1_{\mC}| + \lambda_{j})
\]
where $(s_{j})_{j\in J}$ are finitely many global sections of $H^0(\cX,m\cL)$ which do not have common zeros, $\lambda_{j} \in \mR$, and $\phi$ is invariant under the action of complex conjugation. 
We define $\overline{\cL}^{\an} = (\cL^{\an},\phi^{\an})$ where 
\[
\phi^{\an}\coloneqq m^{-1}\max_{j\in J} (\log |s_{j}| + \lambda_{j})
\]
where $(s_{j})_{j\in J}$ and $\lambda_{j} \in \mR$ are as above. 
Furthermore, it is clear that $\phi^{\an}$ is a tropical Fubini--Study metric on $\cL^{\an}$ (\autoref{defn:tropicalFS}).

In fact, we can easily define an inverse, and hence we have the following result. 

\begin{lemma}\label{lemma:equivGTFS}
Let $\cX$ denote a projective arithmetic variety. 
There is an equivalence of categories
\[
\acPic(\cX)_{\gtFS} \cong \acPic(\cX^{\an})_{\eqv, \tFS}. 
\]
\end{lemma}

\begin{proof}
By the above, we have a map $\aPic(\cX)_{\gtFS} \to \aPic(\cX^{\an})_{\eqv, \tFS}$. 
Conversely, given an element $\overline{\cL}^{\an} = (\cL^{\an},\phi)$, we define an element of $\aPic(\cX)_{\gtFS}$ by setting $(\cL,F_{\iota}^*F_{x_{\infty,1}}^*\phi)$ where $\iota\colon (\mR,|\cdot|) \to (\mC,|\cdot|)$ is the bounded homomorphism of Banach rings defined by $a\mapsto a$, $F_{\iota}^*\colon \cX(\mC) \to \cX_{\sH_{x_{\infty,1}}}^{\an}\cong \cX_{\mR}^{\an}$ is pullback map from Subsection \ref{subsec:basechange}, and $F_{x_{\infty,1}}^*$ is the pullback map form \autoref{example:base_change_metrics}. 
Clearly, these maps are inverses, and it is also clear that isometries are sent to isometries so the claim follows. 
\end{proof}

%
%

Next, we show that our new analytification agrees with that of Yuan and Zhang in the setting of pure tropical Fubini--Study Hermitian line bundles. 
Recall that a pure tropical Fubini--Study metric is one where the real constants $\lambda_{j}$ are all zero. 

This result essentially follows from the following lemma. 

\begin{lemma}\label{lemma:modelmetrics}
Let $K$ be a complete non-Archimedean valued field (possibly trivially valued), let $X$ be a projective $K$-variety, and let $L$ be a semiample line bundle on $X$.

A pure tropical Fubini--Study metric $\phi$ on $L$ (i.e., a metric of the form $\phi \coloneqq m^{-1}\max (\log |s_{j}|)$ where $(s_{j})_{j\in J}$ is a finite collection of global sections of $H^0(X,mL)$ which do not have common zeros) is the same thing as a model metric defined by a semiample model of $L$. 
\end{lemma}

\begin{proof}
This follows from \cite[Proposition 2.13.(ii) \& Lemma 2.14]{BoucksomJonsson:SingularPSH}, in the non-trivial and trivially valued setting, respectively. 
\end{proof}

\begin{lemma}
Let $\cX$ be a projective arithmetic variety, and let $\overline{\cL} = (\cL,\phi)$ be global tropical Fubini--Study metric which is pure. 
Then, $\phi^{\an}_{\YZ} = \phi^{\an}$. 
\end{lemma}

\begin{proof}
This follows from the analytification maps constructed above, \autoref{lemma:modelmetrics}, and the fact that the pullback of a semiample line bundle is again semiample (cf.~\cite[Proposition 1.2]{Fujita:Semipositive}). 
\end{proof}

\section{Adelic divisors and adelic line bundles}
\label{sec:adelicdivisorsandlinebundles}
In this section, we briefly review Yuan and Zhang's definitions of adelic divisors and adelic line bundles on quasi-projective arithmetic varieties and recall their theory of compactified divisors and compactified line bundles.

For the beginning of this section, let $\cU$ denote a quasi-projective arithmetic variety. 
\subsection{Adelic divisors on quasi-projective arithmetic varieties}
Let $\cX$ be a projective model of $\cU$ i.e., there exists an open $\mZ$-immersion $\cU \hookrightarrow \cX$ with $\cX$ being a projective arithmetic variety. 
The first step in defining adelic divisors is to introduce a group of objects with mixed coefficients. 
We refer the reader to \cite[Section 2.3.1]{YuanZhang:AdelicLineBundles} for the definition of arithmetic $\mQ$-divisors $\aDiv(\cX)_{\mQ}$ on projective arithmetic varieties, which we will freely use. 
We remind the reader of the definition of integral arithmetic divisors from \autoref{defn:arithmeticdivisor}. 

\begin{definition}
The \cdef{group of arithmetic $(\mQ,\mZ)$-divisors of $(\cX,\cU)$} is defined as the fibered product
\[
\aDiv(\cX,\cU) \coloneqq \aDiv(\cX)_{\mQ} \times_{\Div(\cU)_{\mQ}} \Div(\cU).
\]
In other words, a $(\mQ,\mZ)$-divisor of $(\cX,\cU)$ is a pair $\underline{\cD} = (\overline{\cD},\cD')$ where $\overline{\cD} \in \aDiv(\cX)_{\mQ}$ and $\cD' \in \Div(\cU)$ such that their image in $\Div(\cU)_{\mQ}$ is the same under the restriction and inclusion maps, respectively. 
We refer to $\overline{\cD}$ as the \cdef{rational part} and $\cD'$ as the \cdef{integral part} of $\underline{\cD}$. 
\end{definition}

Using the notion of effective arithmetic divisors, we can define a partial ordering on $\aDiv(\cX,\cU)$. In particular given two $(\mQ,\mZ)$-divisors of $(\cX,\cU)$, namely $\underline{\cD} = (\overline{\cD},\cD')$ and $\underline{\cE} = (\overline{\cE},\cE')$, we say that $\underline{\cD} \geq \underline{\cE}$ or $\underline{\cE} \leq \underline{\cD}$ in $\aDiv(\cX,\cU)$ if the image of $\overline{\cD}  - \overline{\cE} \in \aDiv(\cX)_{\mQ}$ and the image of $\cD' - \cE' \in \Div(\cU)$ are both effective.

Recall that the collection of projective models $\cX$ of $\cU$ is an inverse system. 
Using pullback morphisms, we can collect all of the arithmetic $(\mQ,\mZ)$-divisors as we vary over all projective models of $\cU$ into a group of so-called model divisors. 

\begin{definition}
The \cdef{group of model divisors} of $\cU$ is defined to be
\[
\aDiv(\cU)_{\model} \coloneqq \varinjlim_{\cX} \aDiv(\cX,\cU)
\]
where $\cX$ ranges over all projective models of $\cU$, and the morphisms are pullback morphisms. 
By the direct limit, we have an induced partial ordering on $\aDiv(\cU)_{\model}$. 
\end{definition}

Roughly speaking, the group of adelic divisors on $\cU$ comes from taking the completion of the group of model divisors with respect to a topology which is induced from a boundary divisor, which we now recall. 
We remind the reader of the conventions concerning arithmetic divisors on projective arithmetic varieties from Subsection \ref{subsec:arithmeticdivisors}. 

\begin{definition}\label{defn:boundarydivisor}
A \cdef{boundary divisor of $\cU$} is a pair $(\cX_0,\overline{\cD}_0)$ consisting of a projective $\mZ$-model $\cX_0$ of $\cU$ and a strictly effective arithmetic divisor $\overline{\cD}_0$ on $\cX_0$ such that the support of $\cD_0$ is equal to $\cX_0\setminus \cU$. 
Note that for any $r\in \mQ$, we may view $r\overline{\cD}_0$ as an element of $\aDiv(\cX_0,\cU)$ by declaring $r\underline{\cD} = (r\overline{\cD},0)$, and hence $r\underline{\cD}$ may also be viewed as an element in $\aDiv(\cU)_{\model}$. 
\end{definition}

\begin{definition}
Let $\underline{\cE} \in \aDiv(\cU)_{\model}$, and let $(\cX_0,\overline{\cD}_0)$ be a boundary divisor for $\cU$. 
The \cdef{boundary norm} $\nrm{\cdot}_{\overline{\cD}_0}$ is defined as
\[
\nrm{\underline{\cE}}_{\overline{\cD}_0} = \inf \{ \varepsilon\in \mQ_{>0} : -\varepsilon\underline{\cD}_0 \leq \underline{\cE} \leq \varepsilon\underline{\cD}_0\}.
\]
Note that if the integral part of $\underline{\cE}$ is non-zero, then such an $\varepsilon$ does not exist and we declare $\inf \emptyset = \infty$. 
As such, the boundary norm is an extended norm. 
\end{definition}

The \cdef{boundary topology} on $\aDiv(\cU)_{\model}$ is the topology induced by the boundary norm $\nrm{\cdot}_{\overline{\cD}_0}$, and by \cite[Lemma 2.5.1]{YuanZhang:AdelicLineBundles}, this topology does not depend on the choice of boundary divisor. 

\begin{definition}
The group of \cdef{adelic divisors} on $\cU$, denoted by $\aDiv(\cU)$, the the completion of $\aDiv(\cU)_{\model}$  with respect to the boundary topology. 
As such, an adelic divisor is represented by a Cauchy sequence in $\aDiv(\cU)_{\model}$ i.e., a sequence $\{ \underline{\cE_i}\}$ in $\aDiv(\cU)_{\model}$ such that there is a sequence $\{ \varepsilon_i\}$ of positive rational numbers converging to zero such that
\[
-\varepsilon_i \underline{\cD}_0 \leq \underline{\cE}_j - \underline{\cE}_i \leq \varepsilon_i\underline{\cD}_0
\]
where $j\geq i\geq 1$. 
The \cdef{class group of adelic divisors} of $\cU$ is defined to be
\[
\aCaCl(\cU) \coloneqq \aDiv(\cU)/\aPr(\cU)_{\model}.
\]
\end{definition}

\subsection{Adelic line bundles on quasi-projective arithmetic varieties}
Next, we define adelic line bundles on quasi-projective arithmetic varieties. 
As above, let $\cU$ be a quasi-projective arithmetic variety. 
First, we recall the notion of model adelic divisors for rational maps.

\begin{definition}
Let $\cX_1$, $\cX_2$ be projective models of $\cU$, and let $\overline{\cL}_i$ be a Hermitian $\mQ$-line bundle on $\cX_i$. 
A \cdef{rational map} $\ell\colon \overline{\cL}_1 \dashrightarrow \overline{\cL}_2$ over $\cU$ is an isomorphism $\ell \colon \cL_{1|\cU} \cong \cL_{2|\cU}$ of $\mQ$-line bundles on $\cU$. 
Let $\cY$ be a projective model of $\cU$, which dominates $\cX_1$ and $\cX_2$ via $\tau_i\colon \cY \to \cX_i$. 
View $\ell$ as a rational section of $\tau_1^*\cL_1^{\vee}\otimes \tau_2^*\cL_2$ on $\cY$ so it defines an arithmetic $\mQ$-divisor $\widehat{\ddiv}_{\cY}(\ell)$ using the metric of $\tau_1^*\overline{\cL}_1^{\vee}\otimes \tau_2^*\overline{\cL}_2$. 
Let $\widehat{\ddiv}(\ell)$ be the image of $\widehat{\ddiv}_{\cY}(\ell)$ in $\aDiv(\cU)_{\model}$ by setting the integral part on $\cU$ to be 0. 
\end{definition}

Let $(\cX_0,\overline{\cD}_0)$ be a boundary divisor, which was defined in \autoref{defn:boundarydivisor}. 

\begin{definition}
We define the objects and morphisms in the \cdef{category of adelic line bundles on $\cU$}. 

The \cdef{objects} in the {category $\acPic(\cU)$ of adelic line bundles on $\cU$} are pairs $(\cL,(\cX_i,\overline{\cL}_i,\ell_i)_{i\geq 1})$ where
\begin{enumerate}
\item $\cL \in \Pic(\cU)$,
\item $\cX_i$ is a projective model of $\cU$,
\item $\overline{\cL}_i $ is an object in $\acPic(\cX_i)_{\mQ}$ i.e., a Hermitian $\mQ$-line bundle on $\cX_i$,
\item $\ell_i\colon \cL \cong \cL_{i|\cU}$ an isomorphism in $\Pic(\cU)_{\mQ}$
\end{enumerate}
which satisfies the following Cauchy condition:~part (4) gives an isomorphism $\ell_i\ell_1^{-1}\colon \cL_{1|\cU}\cong \cL_{i|\cU}$ of $\mQ$-line bundles, and thus a rational map $\ell_i\ell_1^{-1}\colon \overline{\cL}_1 \dashrightarrow \overline{\cL}_i$ over $\cU$ which gives rise to a model divisor $\widehat{\ddiv}(\ell_i\ell_1^{-1})$ in $\aDiv(\cU)_{\model}$, and we require the sequence $\{\widehat{\ddiv}(\ell_i\ell_1^{-1})\}_{i\geq 1}$ to be a Cauchy sequence in $\aDiv(\cU)_{\model}$ under the boundary topology. We will simply denote the objects of $\acPic(\cU)$ by $\aPic(\cU)$. 

The \cdef{morphisms} in {category $\acPic(\cU)$ of adelic line bundles on $\cU$} from an object $(\cL,(\cX_i,\overline{\cL}_i,\ell_i)_{i\geq 1})$ to another object $(\cL',(\cX_i',\overline{\cL}_i',\ell_i')_{i\geq 1})$ is an isomorphism $\iota\colon \cL \to \cL'$ of elements in $\Pic(\cU)$ such that the composition $\ell_i'\iota\ell_i^{-1}\colon \cL_{i|\cU} \to \cL_{i|\cU}'$ induces a rational map $\ell_i'\iota\ell_i^{-1}\colon \overline{\cL}_i \dashrightarrow \overline{\cL}_i'$ which defines a model divisor $\wt{\ddiv}(\ell_i'\iota \ell_i') \in \aDiv(\cU)_{\model}$ whose image in $\Div(\cU)$ is zero. 
Moreover, the sequence 
\[
\{ \widehat{\ddiv}(\ell_i'\iota\ell_i^{-1})\}_{i\geq 1} \in \aDiv(\cU)_{\model}
\]
is required to converge to 0 in the boundary topology i.e., there exists a sequence $\{ \varepsilon_i\}$ of positive rational numbers converging to 0 such that 
\[
-\varepsilon_i \underline{\cD}_0 \leq \widehat{\ddiv}(\ell_i'\iota\ell_i^{-1}) \leq \varepsilon_i\underline{\cD}_0.
\]
\end{definition}

As in the classical case, we have the following result.

\begin{prop}[\protect{\cite[Proposition 2.6.1]{YuanZhang:AdelicLineBundles}}]\label{prop:isomclassadelic}
Let $\cU$ be a quasi-projective arithmetic variety. There is a canonical isomorphism
\[
\aCaCl(\cU) \cong \aPic(\cU).
\]
\end{prop}

\section{Proof of \autoref{thmx:main0}}
\label{sec:proofs}
In this section, we prove our main result. 

\subsection{Strongly semiample adelic line bundles on quasiprojective varieties}
First, we introduce our new definitions. 
We remind the reader of the definition of a global tropical Fubini--Study metric from \autoref{defn:GTFSH}. 

\begin{definition}\label{defn:ssalinebundles}
Let $\cU$ be a quasi-projective arithmetic variety. 
\begin{enumerate}
\item We say that an adelic line bundle $\overline{\cL} \in \acPic(\cU)$ is \cdef{strongly semiample} if it is isomorphic to an object $(\cL,(\cX_i,\overline{\cL}_i,\ell_i)_{i\geq 1})$ where each $\overline{\cL}_i$ is a global tropical Fubini--Study Hermitian line bundle i.e., a line bundle $\cL_i$ equipped with a global tropical Fubini--Study metric $\nrm{\cdot}_i$ on $\cX_i$. 
We denote the subcategory of strongly semiample adelic line bundles by $\acPic(\cU)_{\ssa}$. 
\item We say that an adelic line bundle $\overline{\cL} \in \acPic(\cU)$ is \cdef{semiample} if there exists a strongly semiample adelic line bundle $\overline{\cM} \in \acPic(\cU)$ such that $a\overline{\cL} + \overline{\cM}$ of strongly semiample for all positive integers $a$. 
\end{enumerate}
\end{definition}

\begin{definition}\label{defn:ssadivisors}
Let $\cU$ be a quasi-projective arithmetic variety. 
We define the \cdef{class group of strongly semiample arithmetic divisors on $\cU$} to be the preimage of $\aPic(\cU)_{\ssa}$ under the isomorphism from \autoref{prop:isomclassadelic}, and we will denote this group by $\aCaCl(\cU)_{\ssa}$. 
\end{definition}

\begin{remark}
The wording ``strongly semiample'' comes from Yuan and Zhang's definition of a strongly nef adelic line bundle \cite[Definition 2.6.2]{YuanZhang:AdelicLineBundles}, which is one where the Hermitian line bundles $\overline{\cL}_i$ in the definition of an adelic line bundle are required to be nef in the sense that they admit a continuous semipositive metric (on the complex points) and have non-zero arithmetic degree on any 1-dimensional closed subvariety of $\cX_i$.

Our notion of strongly semiample requires an a priori stronger condition on the metric at the Archimedean place, namely that the line bundle is semiample which implies the existence of a continuous semipositive metric. However, we do not impose any Arakelov condition.
It would be very interesting to find a Berkovich analytic characterization of the arithmetic degree condition. 
\end{remark}

\begin{remark}\label{remark:semiamplemetrized}
In \cite[Sections 3 \& 4]{Zhang:PositiveArithmeticVarieties}, Zhang defined the notion of semiample metrized adelic line bundle on projective arithmetic varieties. 
Using \cite[Theorem 3.5]{Zhang:PositiveArithmeticVarieties} and \cite[Theorem 0.2]{Moriwaki:SemiampleHermitian}, we note that a semiample line bundle $\cL$ with ample generic fiber on a projective arithmetic variety $\cX$ will give rise to semiample metrized adelic line bundle. 
\end{remark}

\subsection{Analytification of strongly semiample adelic line bundles on quasi-projective arithmetic varieties}
First, we determine the metric induced by the analytification of a strongly semiample adelic line bundle. 
In the course of our proof, we will describe the analytifcation of an adelic line bundle on a quasi-projective arithmetic variety, and for our purposes, we will use the analytification map in the projective setting as defined in Subsection \ref{subsubsec:analytificationnew}. 

\begin{lemma}\label{lemma:analytification_ssa_eqvsemip}
Let $\cU$ be a quasi-projective arithmetic variety, and let $\overline{\cL}$ be a strongly semiample adelic line bundle on $\cU$. 
Then, $\overline{\cL}^{\an}$ is a metrized line bundle on $\cU^{\an}$ equipped with a norm-equivariant and continuous semipositive metric. 
Moreover, there is a natural transformation
\[
\acPic(\cU)_{\ssa} \to \acPic(\cU^{\an})_{\eqv,\semip}
\]
\end{lemma}

\begin{proof}
The proof will essentially follow from the definition of analytification of an adelic line bundle on a quasi-projective arithmetic variety, which we now recall.

Let $(\cL,(\cX_i,\overline{\cL}_i,\ell_i)_{i\geq 1})$ be an object of $\acPic(\cU)_{\ssa}$. 
Note that $\overline{\cL}_i = (\cL_i,\phi_i)$ is a global tropical Fubini--Study Hermitian line bundle on $\cX_i$. 
By \autoref{lemma:equivGTFS}, we know that $\phi_i^{\an}$ is a norm-equivariant, tropical Fubini--Study metric on $\cL_i^{\an}$. 
Let $\phi_{i|\cU^{\an}}^{\an}$ denote the metric of $\cL \cong \cL_{i|\cU}$ over $\cU^{\an}$ induced by $(\cX_i,\overline{\cL}_i)$, and note that $\phi^{\an}_{i|\cU^{\an}}$ is a tropical Fubini--Study metric on $\cU^{\an}$ by \cite[Proposition 2.14.(6)]{PilleSchneider:Global}. 
As in \cite[Proof of Proposition 3.4.1:~quasi-projective case]{YuanZhang:AdelicLineBundles}, we see that $\phi^{\an}_{i|\cU^{\an}}$ converges pointwise to a continuous metric $\phi_{\cU^{\an}}$, and moreover, $\phi^{\an}_{i|\cU^{\an}}$ is compactly convergent to $\phi_{\cU^{\an}}$. 
The first statement now follows from our definition of continuous semipositive metric (\autoref{defn:continuous_semipositive}). To conclude, we note that in \cite[Proof of Proposition 3.4.1:~quasi-projective case]{YuanZhang:AdelicLineBundles}, the authors show that the analytification map is fully faithful, and hence the result follows.
\end{proof}

Now, we need to show that a norm-equivariant, continuous semipositive metric on a line bundle $\cL$ on $\cU^{\an}$ comes from a strongly semiample adelic line bundle on $\cU$. 
Recall the definitions and notation from \autoref{defn:classgroup_eqvsemip_Berk} and \autoref{defn:ssadivisors}. 
We can rephrase our desired statement as follows:

\begin{prop}\label{prop:isom_ssadivisor_semipositive}
Let $\cU$ be a quasi-projective arithmetic variety. There is an isomorphism
\[
\aCaCl(\cU)_{\ssa} \cong \aCaCl(\cU^{\an})_{\eqv,\semip}. 
\]
\end{prop}

By class group isomorphisms and \autoref{lemma:analytification_ssa_eqvsemip}, we have a map
\[
\aCaCl(\cU)_{\ssa} \to \aCaCl(\cU^{\an})_{\eqv,\semip},
\]
so it suffices to define an inverse map
\[
\aCaCl(\cU^{\an})_{\eqv,\semip} \to \aCaCl(\cU)_{\ssa}. 
\]

Before defining the map, we need to understand the Green's functions on a representative of $\aCaCl(\cU^{\an})_{\eqv,\semip}$. 

\begin{lemma}\label{lemma:Greensfunctionsemipositive}
Let $(\cD,g_{\cD})$ be an arithmetic divisor on $\cU^{\an}$, which is a representative of an element of $\aCaCl(\cU^{\an})_{\eqv,\semip}$. 
There exists a net of projective models $(\cX_i)$ and arithmetic divisors $(\cE_i,g_{\cE_i}) \in \aDiv(\cX_i^{\an})_{\eqv, \tFS}$ such that $\cE_{i|\cU^{\an}} = \cD$, $g_{\cE_i|\cU^{\an}}$ is a Green's function for $\cD$ for each $i$,  and the net $(g_{\cE_i|\cU^{\an}})$ compactly converges to $g_{\cD}$. 
%
\end{lemma}
\begin{proof}
By \autoref{defn:Berk_tfs_semi_arithmeticdivisors}, a representative of $\aCaCl(\cU^{\an})_{\eqv,\semip}$ corresponds to an arithmetic divisor $(\cD,g_{\cD})$ on $\cU^{\an} $ such that the metric $\nrm{\cdot}_{g_{\cD}}$ associated to $g_{\cD}$ is a continuous semipositive metric on $\cO(\cD)$. 
The definition of continuous semipositive metric (\autoref{defn:continuous_semipositive}) implies that there exists a net $(\cX_i,\cL_i,\nrm{\cdot}_i)$ where $(\cX_i,\cL_i)$ is a net of projective models for $(\cU,\cO(\cD))$ and $\nrm{\cdot}_i$ is a tropical Fubini--Study metric on $\cL_i^{\an}$ such that the restriction of the metrics $\nrm{\cdot}_{i|\cU^{\an}}$ converges compactly to $\nrm{\cdot}_{g_{\cD}}$. 
Since $(\cL_i^{\an},\nrm{\cdot}_i) \in \aPic(\cX_i^{\an})_{\eqv,\tFS}$, we can find arithmetic divisors $(\cE_i,g_{\cE_i})$ where $\cE_i = \ddiv(s_i)$ and $g_{\cE_i} = -\log \nrm{s_i}_{i}$ for $s_i$ a rational section of $\cL_i$ is tropical Fubini--Study.  
Moreover, the restriction of each $g_{\cE_i}$ gives a Green's function for $\cE_{i|\cU} = \cD$, and the above compact convergence of metrics implies that $g_{\cE_i|\cU^{\an}}$ compactly converge to $g_{\cD}$. 
\end{proof}

Returning to our goal of constructing an inverse map, 
let $(\cX_0,\overline{\cD}_0)$ be a boundary divisor for $\cU$ (\autoref{defn:boundarydivisor}). 
Using \cite[Lemma 3.3.2]{YuanZhang:AdelicLineBundles}, Green's function $g_{\cD_0}$ of the strictly effective arithmetic divisor $\overline{\cD}_0$ can be analytified to an analytic Green's function on $\cX^{\an}$, which induces a Green's function on $\cU^{\an}$, which we will denote by $\wt{g}_{\cD_0}$. 
Moreover, the construction of the analytification implies the following result. 
We remind the reader that $\cU^{\beth}$ is the interior of $\cU^{\an}$ as defined in \autoref{defn:interiorboundary}. 

\begin{prop}[\protect{\cite[Proposition 3.2]{Song:EquivariantAdelic}}]\label{prop:Greens_nonvanishing}
For any $x\in \cU^{\an}$, $\wt{g}_{\cD_0}(x)\geq 0$ and $\wt{g}_{\cD_0}(x) = 0$ if and only if $x\in \cU^{\beth}$.
\end{prop}

We now recall a construction of Song which gives an alternative description of the normalized boundary $\wt{\cU}^{b}$ of $\cU^{\an}$ (\autoref{defn:interiorboundary}) using properties of the Green's function $\wt{g}_{\cD_0}$. 

\begin{prop}[\protect{\cite[Proposition 3.4 \& Corollary 3.5]{Song:EquivariantAdelic}}]\label{prop:alternative_normalized_boundary}
Keep the notation as above.

The subspace $\delta_{\overline{\cD}_0}(\cU)$ of $\cU^{\an}$ defined via
\[
\delta_{\overline{\cD}_0}(\cU) = \{ x\in \cU^b : \wt{g}_{\cD_0}(x) = 1 \} \cup \{ x \in \cU^b : \max_{y \sim x}\, \wt{g}_{\cD_0}(y) \leq 1 \textrm{ and }\wt{g}_{\cD_0}(x) = \max_{y\sim x}\, \wt{g}_{\cD_0}(y)\}.
\]
is a compact Hausdorff subspace of $\cU^{\an}$ which is homeomorphic to $\wt{\cU}^b$. 
\end{prop}

We now define our inverse map. Our construction and proof follows from the same ideas as \cite{Song:EquivariantAdelic}. 

\begin{prop}\label{prop:inversemap_semipositive_ssa}
Let $\cU$ be a quasi-projective arithmetic variety. 
There is a map
\[
\aCaCl(\cU^{\an})_{\eqv,\semip} \to \aCaCl(\cU)_{\ssa}
\]
which is inverse to the map 
\[
\aCaCl(\cU)_{\ssa}  \to \aCaCl(\cU^{\an})_{\eqv,\semip}.
\]
\end{prop}

\begin{proof}
Using properties of Green's functions, it suffices to construct our map when $\cD = 0$. 
In other words, we need to show that a norm-equivariant, continuous semipositive Green's function $g$ on $\cU^{\an}$ associated to the trivial divisor comes from the Green's function of an element in $\aCaCl(\cU)_{\ssa}$ with underlying divisor being the trivial divisor.

From \autoref{lemma:Greensfunctionsemipositive}, there exists a net $\{ g_{\cE_i}\}$ of norm-equivariant,  tropical Fubini--Study Green's functions associated to the trivial divisor on $\cU^{\an}$ such that $g_{\cE_i}$ compactly converges to $g$. 
Here we have suppressed the restriction notation. 
We recall that these Green's functions come from a net of semiample projective models $(\cX_i,\overline{\cL}_i)$ of $(\cU,0)$. 
Fix a boundary divisor $(\cX_0,\overline{\cD}_0)$ of $\cU$, and let $\wt{g}_{0}$ be the corresponding analytic Green's function.

By \autoref{prop:Greens_nonvanishing}, we have that $h_{\cE_i} = g_{\cE_i}/\wt{g}_0$ and $h = g/\wt{g}_0$ are continuous, norm-equivariant function on $\cU^{b}$, and hence continuous functions on $\wt{\cU}^b$. 
Using \autoref{prop:alternative_normalized_boundary}, we have that $h_{\cE_i}$ and $h$ also induce a continuous function on the compact Hausdorff subspace $\delta_{\cD_0}(\cU)$. 
By the above compact convergence, we have that $h_{\cE_{i|\delta_{\cD_0}(\cU)}}$ converges uniformly to $h_{|\delta_{\cD_0}(\cU)}$. 
As such, there exists a sequence of rational numbers $\{ \varepsilon_i\}$ converging to zero such that for all $j>i$, we have
\[
|h_{\cE_i}(x) - h_{\cE_j}(x)| \leq \varepsilon_i
\]
for all $x\in \delta_{\cD_0}(\cU)$. 
Therefore, we have that 
\[
-\varepsilon_i \wt{g}_{0}(x) \leq g_{\cE_i}(x) - g_{\cE_j}(x) \leq \varepsilon_i \wt{g}_{0}(x)
\]
for all $j>i$ and all $x\in \delta_{\cD_0}(\cU)$ or equivalently for all $x\in \cU^b$ by norm-equivariance and \autoref{prop:alternative_normalized_boundary}. 
Note that the functions $g_{\cE_i}$ and $\wt{g}_{0}$ all vanishing on $\cU^{\beth}$ by \autoref{prop:Greens_nonvanishing} and hence these above inequalities hold for all $x\in \cU^{\an}.$

By \cite[Lemma 3.3.3]{YuanZhang:AdelicLineBundles}, after possibly normalizing, we have the inequalities
\[
-\varepsilon_i \overline{\cD}_0 \leq \overline{\cE}_i - \overline{\cE}_j \leq \varepsilon_i \overline{\cD}_0.
\]
Therefore, $\overline{\cE}_i$ converges to an adelic divisor $\overline{\cE}$.
Moreover, as with Song's argument, we see that the Green's function associated to $\overline{\cE}$ is $g$, and hence $(\overline{\cE},g)$ is an element of $\aCaCl(\cU)_{\ssa}$. 
Furthermore, it follows from the construction that this is in fact the inverse of the map from \autoref{lemma:analytification_ssa_eqvsemip}. 
\end{proof}

\begin{proof}[Proof of \autoref{prop:isom_ssadivisor_semipositive}]
This follows from \autoref{lemma:analytification_ssa_eqvsemip} and \autoref{prop:inversemap_semipositive_ssa}. 
\end{proof}

Using the above results, we have our main theorem. 

\begin{theorem}[=\autoref{thmx:main0}]
Let $\cU$ be a quasi-projective arithmetic variety. 
There is an equivalence of categories
\[
\acPic(\cU)_{\ssa} \cong \acPic(\cU^{\an})_{\eqv, \semip}.
\]
\end{theorem}

\begin{proof}
Note that \autoref{prop:isom_ssadivisor_semipositive} and \autoref{prop:inversemap_semipositive_ssa} establish a bijection between the objects in these categories, and fully faithfulness follows from \cite[Proposition 3.4.1]{YuanZhang:AdelicLineBundles}. 
\end{proof}

\section{Applications}
\label{sec:applications}
In this section, we discuss several applications of our constructions and of \autoref{thmx:main0}.

\subsection{Families of Monge--Amp\`ere measures}
\label{subsec:familiesMA}
Let $\cX$ be a projective arithmetic variety of dimension $n$, and let $\cX^{\an} \to \sM(\mZ)$ denote the Berkovich analytification. 
In \cite[Section 4.3]{PilleSchneider:Global}, Pille-Schneider described how one can define a family, as $x$ varies over $\sM(\mZ)$, of Monge--Amp\`ere measures to a finite collection of semiample line bundles $\cL_1^{\an},\dots,\cL_n^{\an}$ equipped with continuous plurisubharmonic (in the sense of \autoref{defn:psh}) metrics $\phi_1,\dots,\phi_n$. We will show how to extend this to the setting where $\cL_i^{\an}$ are semiample line bundles equipped with a continuous semipositive metric (in the sense of \autoref{defn:continuous_semipositive}) on a quasi-projective arithmetic variety.

Let $\cU$ be a quasi-projective arithmetic variety of dimension $n$, let $\pi\colon \cU^{\an} \to \sM(\mZ)$ denote the Berkovich analytification, and let $\cL_1^{\an},\dots,\cL_n^{\an}$ be semiample line bundles on $\cU^{\an}$ equipped with continuous semipositive metrics $\phi_1,\dots,\phi_n$.

Our goal is to define the Monge--Amp\`ere measure fiberwise. 
Recall that for each $\phi_i$, there exists a net of models $(\cX_{i,j},\overline{\cL}_{i,j})_{j\in J}$ such that each $\overline{\cL}_{i,j} = (\cL_{i,j},\phi_{i,j}) \in \aPic(\cX_{i,j}^{\an})_{\eqv,\tFS}$ where $\phi_{i,j}$ compactly converge to $\phi_i$ on $\cU^{\an}$. 
Note that since $\cU$ is quasi-projective, \cite[\href{https://stacks.math.columbia.edu/tag/0B3G}{Tag 0B3G}]{stacks-project} and \autoref{prop:Berkovichproperties}.(1) implies that for any point $x\in \sM(\mZ)$, $\cU^{\an}_{\sH_x}$ is locally compact. 
As such, we see that for each $x\in \sM(\mZ)$, the convergence of $\phi_{i,j,x}$ to $\phi_{i,x}$ is locally uniform. 


First, we describe the construction over Archimedean points of $\sM(\mZ)$. 
By \autoref{lemma:proj_restrictedcontinuoussemipositive} and  \autoref{thm:eqiuvalentcomplexHermitianline}, we have that for a tropical Fubini--Study metric $\phi$ on a semiample line bundle $\cL^{\an}$ over the analytification $\cX^{\an}$ of a projective arithmetic variety $\cX$, the pullback $\phi_x$ of $\phi$ along an Archimedean point $x\in \sM(\mZ)$ is a continuous semipositive metric on $\cL(\mC)$ that is invariant under complex conjugation. 
As such, we can use the above observation concerning the locally uniform convergence to define Monge--Amp\`ere measures over these points via constructions from \cite[Theorem 2.1]{BedforTaylor:CapacityPSH} or \cite[Corollary 1.6]{Demailly:MAOperators}.

Over non-Archimedean points $x\in \sM(\mZ)$, we could follow the strategy of \cite[Section 3.6.7]{YuanZhang:AdelicLineBundles}, which we will briefly outline. 
First, we note that if $x$ corresponds to a non-trivial valuation, then one can show that the pullback metrized line bundle $(\cL_x^{\an},\phi_{i,x})$ is a strongly nef compactified metrized line bundle\footnote{In this work, we will not define strongly nef compactified metrized line bundles as they are not a relevant concept for us. We refer the reader to \cite[Section 3.6.6]{YuanZhang:AdelicLineBundles} for details.} over $\cU^{\an}_{\sH_x}$ via the same argument as \cite[Lemma 3.6.7]{YuanZhang:AdelicLineBundles}. As such, we can define the Monge--Amp\`ere measure as they do in this setting. 
For our purposes, we want to define the Monge--Amp\`ere measure on the fiber of \textit{any} non-Archimedean point of $\sM(\mZ)$. 

For a general non-Archimedean point $x\in \sM(\mZ)$,  the previous observation concerning the locally uniform convergence tells us that it suffices to study the tropical Fubini--Study metric setting i.e., when $\phi_{i,x}$ is a tropical Fubini--Study metric. 
In this case, we have that $\phi_{i,x}$ is locally psh-approachable (cf.~\cite[Corollaire 6.3.4]{ChambertLoirDucros:FormesDifferentielles} or \cite[Lemma 2.9]{BoucksomJonsson:SingularPSH}). 
As such,  \cite[Corollaire 5.6.5]{ChambertLoirDucros:FormesDifferentielles} shows that there is a canonical measure
\[
c_1(\overline{\cL}_{1,x}^{\an})\cdots c_1(\overline{\cL}_{n,x}^{\an}) = d'd''(-\log\nrm{\cdot}_{1,x})\wedge \cdots \wedge d'd''(-\log\nrm{\cdot}_{n,x}),
\]
over $\cU^{\an}_{\sH_x}$ which is defined by the following weak convergence process.

For any $i = 1,\dots,n$,  \autoref{lemma:proj_restrictedcontinuoussemipositive} implies that $\phi_{i,x}$ is a continuous semipositive metric in the appropriate setting. 
Furthermore, the uniform description of such metrics implies that $\phi_{i,x}$ on $\cL_{i,x}^{\an}$ is defined as the limit of tropical Fubini--Study metrics induced by projective models $(\cX_{i,j,\sH_x},\cL_{i,j,x})$ of $(\cU_{\sH_x},\cL_{i,x})$ over $\Spec(\sH_x)$. 
We may and do assume that $\cX_{i,j,\sH_x}$ is independent of $i$, and simply write this as $\cX_{j,\sH_x}$. Also, denote by $\overline{\cL}_{i,j,x} = (\cL_{i,j,x},\phi_{i,j,x})$ the metrized line bundle as above where each $\phi_{i,j,x}$ is a tropical Fubini--Study metric. 
Let $C_c(\cU^{\an}_{\sH_x})$ denote the space of real-valued continuous and compactly supported functions on $\cU^{\an}_{\sH_x}$. 
Then, this construction gives that for any $f\in C_c(\cU^{\an}_{\sH_x})$,
\begin{equation}\label{eqn:MAmeasure}
\int_{\cU^{\an}_{\sH_x}} f\, c_1(\overline{\cL}^{\an}_{1,x})\cdots c_1(\overline{\cL}^{\an}_{n,x})\coloneqq \lim_{j\to \infty} \int_{\cX_{j,\sH_x}^{\an}} f\, c_1(\overline{\cL}_{1,j,x}^{\an})\cdots c_1(\overline{\cL}_{n,j,x}^{\an}). 
\end{equation}
Since $\cX_{j,\sH_x}$ is projective over $\Spec(\sH_x)$, the right hand side is equal to the integration defined by the global intersection numbers by \cite{ChambertLoirDucros:FormesDifferentielles}. 

Now following \cite{PilleSchneider:Global}, we define a family of Monge--Amp\`ere measures. 

\begin{definition}\label{defn:MAmeasure_quasiproj}
Let $\cU$ be a quasi-projective arithmetic variety of dimension $n$, and let $\overline{\cL}_1^{\an},\dots, \overline{\cL}_n^{\an}$ be semiample line bundles on $\cU^{\an}$ equipped with continuous semipositive metrics.  
We define the associated family of Monge--Amp\`ere measures as follows:~for any point $x\in \sM(\mZ)$, let $F_x\colon \cU_{\sH_x}^{\an} \to \cU^{\an}$ denote the base change map. 
We set
\[
(c_1(\overline{\cL}_1^{\an})\cdots c_1(\overline{\cL}_n^{\an}))_x \coloneqq (F_x)_*(c_1(\overline{\cL}^{\an}_{1|\cU_{\sH_x}^{\an}})\cdots c_1(\overline{\cL}^{\an}_{n|\cU_{\sH_x}^{\an}}))
\]
to be the pushforward measure of the measure from \eqref{eqn:MAmeasure} along $F_x$ (cf.~\autoref{notion:basechange}). 
\end{definition}

\begin{remark}
By \autoref{lemma:proj_continuoussemipositive_continuouspsh}, our definition of continuous semipositive metric and the definition of continuous plurisubharmonic agree on a projective arithmetic variety, and so this new construction will recover that of \cite[Definition 4.8]{PilleSchneider:Global}. 
\end{remark}

\subsection{Invariant adelic line bundles}
\label{subsec:invariant_adelic}
Next, we will prove that the invariant adelic line bundle associated to a polarized dynamical system over a flat and quasi-projective arithmetic variety is a strongly semiample line bundle (\autoref{defn:ssalinebundles}), and hence the analytification is a norm-equivariant, continuous semipositive metrized line bundle on the associated Berkovich space. 

First, we need a definition and a construction.

\begin{definition}\label{defn:PDS}
A \cdef{polarized dynamical system} over an integral scheme $S$ is a triple $(X,f,L)$ such that:
\begin{enumerate}
\item $X$ is an integral, projective, and flat $S$-scheme,
\item $f\colon X \to X$ is a morphism over $S$,
\item $L\in \Pic(X)_{\mQ}$ is a $\mQ$-line bundle on $X$ that is relatively ample over $S$ such that $f^*L \cong qL$ for some rational number $q>1$. 
\end{enumerate}
\end{definition}

Next, we recall Yuan and Zhang's construction of the $f$-invariant adelic line bundle from \cite[Section 6.1.1]{YuanZhang:AdelicLineBundles}. 
Let $S$ be a quasi-projective arithmetic variety, and let $(X,f,L)$ be a polarized dynamical system over $S$. 
Choose a projective model $\cX \to \cS$ of $X\to S$ and a $\mQ$-Hermitian line bundle $\overline{\cL} = (\cL,\nrm{\cdot})$ such that $(\cX_S,\cL_S) \cong (X,L)$. 
For each positive integer $i$, consider the composition $X \xrightarrow{f_i} X \to \cX$, and let $f_i\colon \cX_i \to \cX$ denote the normalization of this map. Denote by $\pi_i\colon \cX_i \to \cS$ the induced map and let $\overline{\cL}_i = q^{-i}f_i^*\overline{\cL}$, which is an element in $\aPic(\cX_i)_{\mQ}$.

In \cite[Section 6.1.1]{YuanZhang:AdelicLineBundles}, the authors enhance the tuple $(\cX_i,\overline{\cL}_i)$ to form an adelic line bundle in $\acPic(X)_{\mQ}$ as follows. 
Note that $X$ is quasi-projective over $\Spec(\mZ)$ by \cite[\href{https://stacks.math.columbia.edu/tag/0C4M}{Tag 0C4M}]{stacks-project}, and that the morphism $f\colon X \to X$ extends to a morphism $f_{\cX}\colon \cX \to \cX$ such that the isomorphism $f^*L \to qL$ also extends to an isomorphism $f_{\cX}^*\cL \to q\cL$. 
By considering the isomorphism
\[
\ell\colon \cL \to q^{-1}f_{\cX}^*\cL
\]
in $\Pic(X)_{\mQ}$ and successively applying $q^{-1}f_{\cX}^*$ to $\ell$, we obtain canonical isomorphisms $\ell_i\colon \cL \to \cL_i$ in $\Pic(X)_{\mQ}$. 
Yuan and Zhang prove that the extra data of this sequence of rational maps $(\ell_i)$ induces an adelic line bundle. 

\begin{theorem}[\protect{\cite[Theorem 6.1.1]{YuanZhang:AdelicLineBundles}}]
\label{thm:invariantadelic}
Let $S$ be a quasi-projective arithmetic variety, and let $(X,f,L)$ be a polarized dynamical system over $S$. 
Fix an isomorphism $f^*L \to qL$ in $\Pic(X)_{\mQ}$ with $q>1$. The above sequence 
\[
(\cL,(\cX_i,\overline{\cL}_i,\ell_i)_{i\geq 1})
\]
converges to an object $\overline{L}_f$ in $\acPic(X)_{\mQ}$. 
The adelic line bundle $\overline{L}_f$ is uniquely determined by the tuple $(S,X,f,L)$ over $\mZ$ and $f^{*}L \to qL$, up to isomorphism, and is $f$-invariant in the sense that $f^*\overline{L}_f \cong q\overline{L}_f$ in $\aPic(X)_{\mQ}$. 
\end{theorem}

Our contribution to this construction is the following. 

\begin{prop}\label{prop:invariant_ssa}
Keep the notation from \autoref{thm:invariantadelic}. 
The $f$-invariant adelic line bundle $\overline{L}_f$ is semiample (\autoref{defn:ssalinebundles}). 
If $S$ has an affine quasi-projective model over $\mZ$, then $\overline{L}_f$ is strongly semiample. 
\end{prop}

\begin{proof}
For the second claim, note that when $S$ has an affine quasi-projective model over $\mZ$, the relative ampleness of $L$ over $S$ implies that $L$ is ample on $X$ by \cite[\href{https://stacks.math.columbia.edu/tag/01VK}{Tag 01VK}]{stacks-project}. 
If we are able to choose $(\cX,\cL)$ such that ${\cL}$ is semiample on $\cX$, then every $\cL_i$ is semiample on $\cX_i$ by pullback.
As each semiample line bundle can be endowed with a global tropical Fubini--Study metric, we would have that $\overline{L}_f$ is indeed strongly semiample. 
The existence of a projective model $(\cX,\cL)$ for $(X,L)$ over $\mZ$ such that $\cL$ is (semi)ample follows from \cite[\href{https://stacks.math.columbia.edu/tag/01PZ}{Tag 01PZ}]{stacks-project} and \cite[\href{https://stacks.math.columbia.edu/tag/01Q1}{Tag 01Q1}]{stacks-project}.

The first claim follows from the proof of \cite[Theorem 6.1.1]{YuanZhang:AdelicLineBundles} and noting an ample $\mQ$-line bundle $\cL$ on a projective model $\cX$ of $X$ over $\Spec(\mZ)$ will have a global tropical Fubini--Study metric. 
\end{proof}

As an immediate corollary of \autoref{thmx:main0}, we have the following. 

\begin{corollary}\label{coro:analytic_invariant_ssa}
Keep the notation from \autoref{thm:invariantadelic}. 
If $S$ has an affine quasi-projective model over $\mZ$, then the analytification of $\overline{L}_f$ is a metrized line bundle on $X^{\an}$  equipped with a norm-equivariant, continuous semipositive metric. 
\end{corollary}

\subsection{Non-degeneracy criterion over trivially valued field}
\label{subsec:Nondegeneracy_trivially}
Finally, building off Subsections \ref{subsec:familiesMA} and \ref{subsec:invariant_adelic}, we now give a criterion of a subvariety of a variety admitting a polarized dynamical system to be non-degenerate.

First, we recall the definition of a non-degenerate subvariety. 
Let $S$ be an affine quasi-projective arithmetric variety, let $(X,f,L)$ be a polarized dynamical system over $S$, let $\overline{L}_f$ denote the invariant adelic line bundle from \autoref{thm:invariantadelic}, and let $Y$ be a closed subvariety of $X$. 
Let $\overline{M} \coloneqq \overline{L}_{f|Y}$ denote the image of $\overline{L}_f$ under the pullback map
\[
\aPic(X)_{\ssa,\mQ} \to \aPic(Y)_{\ssa,\mQ}. 
\]
Following \cite[Section 6.2.2 \& Lemma 5.4.4]{YuanZhang:AdelicLineBundles}, we say that $Y$ is \cdef{non-degenerate} if for any \textit{place} of $\mQ$ i.e., point $x_{v,1}$ in $\sM(\mZ)$ where $v$ is either a prime number $p$ or $\infty$, we have that 
\[
\int_{Y_{\sH_{x_{v,1}}}} c_1(F_{{x_{v,1}}}^*\overline{L}_{f}^{\an})^{\dim Y} \neq 0
\]
or equivalently, 
\[
c_1(F_{{x_{v,1}}}^*\overline{L}_{f}^{\an})_{|Y_{\sH_{x_{v,1}}}^{\an}}^{\dim Y} \neq 0. 
\]
Yuan and Zhang actually define the notion of non-degenerate using an Arakelov intersection number, however, they show in \cite[Lemma 5.4.4]{YuanZhang:AdelicLineBundles} that the global intersection number can be computed locally as above. 
Their proof crucially uses a global intersection formula \cite[Th\'eor\`eme 1.4]{ChambertLoirThuillier:MahlerMeasure}. 

Recently, Guo \cite{Guo:Integration} showed how to compute this intersection number in a purely local manner over a complete \textit{non-trivially} valued field.
The main obstacle is that the constant function 1 is not compactly supported on the analytification of a quasi-projective arithmetic variety so the result is not immediate. Guo's idea was to use certain cut-off functions, namely difference of plurisubharmonic functions, to relate the weak limit process to a limit of tradition intersection numbers. 

In Subsection \ref{subsec:familiesMA}, we illustrated how one could define Monge--Amp\`ere measures for the pullback of a strongly semiample line bundle $\cL$ over a quasi-projective arithmetic variety $\cU$ along any point of $x\in \sM(\mZ)$, in particular a point corresponding to a trivial valuation. 
This construction was not available previously due to the fact that compactified strongly nef adelic divisors and line bundles were not defined over trivially valued fields. 

We remark that Guo's proof carries over to the setting of an \textit{infinite}, trivially valued non-Archimedean field \textit{mutatis mutandis}\footnote{A few words are in order. First, it is necessary to assume that we work over an infinite field in order to use \cite[Lemma 3.1]{Guo:Integration}. Moreover, Guo's proof uses the theory of differential forms and differential analysis on Berkovich spaces, which was developed in full generality in \cite{ChambertLoirDucros:FormesDifferentielles}. As this theory is valid in the trivially valued setting, Guo's arguments carry over immediately to our setting.} using the weak limit process. 
As such, we have the following minor strengthening and reinterpretation of Guo's result. 

\begin{theorem}[Guo $+\, \varepsilon$]\label{thm:intersection_Guo_epsilon}
Let $\cU$ be a quasi-projective arithmetic variety of dimension $n$. 
For any $n$-tuple of strongly semiample adelic line bundles $\overline{\cL}_1,\dots,\overline{\cL}_n$ and any $x\in \sM(\mZ)^{\eta}$, we have
\begin{equation}\label{eqn:totalmass}
\int_{\cU^{\an}} (c_1(\overline{\cL}_1^{\an})\cdots c_1(\overline{\cL}_n^{\an}))_x = \pwr{ \cL_{1|\cU_{\sH_x}} \cdot \cL_{2|\cU_{\sH_x}} \cdots \cL_{n|\cU_{\sH_x}}}
\end{equation}
where the left hand side is defined as in \autoref{defn:MAmeasure_quasiproj} and the right hand side is computed via a limit of usual intersection numbers as  in \cite[Proposition 4.1.1]{YuanZhang:AdelicLineBundles}.   
\end{theorem}

\begin{proof}
It suffices to prove that for each $x\in \sM(\mZ)^{\eta}$, we have that 
\[
\int_{\cU^{\an}_{\sH_x}}  (c_1(\overline{\cL}^{\an}_{1|\cU_{\sH_x}^{\an}})\cdots c_1(\overline{\cL}^{\an}_{n|\cU_{\sH_x}^{\an}})) = \pwr{ \cL_{1|\cU_{\sH_x}} \cdot \cL_{2|\cU_{\sH_x}} \cdots \cL_{n|\cU_{\sH_x}}}
\]
and that this quantity is independent of $x$.

The equality claim follows from \cite[Theorem 1.2]{Guo:Integration} with the observation that his arguments carry over to the setting of an infinite, trivially valued field.  
The fact that the quantity is independent of $x\in \sM(\mZ)^{\eta}$ follows because the right hand side is independent of $x$ as it is defined via a limit of usual intersection numbers on projective $\sH_x$-varieties and this latter quantity is independent of $x$ by flatness of the projective $\mZ$-models in the definition of strongly semiample adelic line bundle. 
\end{proof}

%
%
%
%

The (potential) utility of \autoref{thm:intersection_Guo_epsilon} is that it allows one to give a criterion for a subvariety of a polarized dynamical system to be non-degenerate involving analytifications over trivially valued fields.

\begin{prop}\label{prop:nondegenerate_trivially}
Let $S$ be an affine quasi-projective arithmetric variety, let $(X,f,L)$ be a polarized dynamical system over $S$, and let $\overline{L}_f$ denote the invariant adelic line bundle from \autoref{thm:invariantadelic}, whose analytification lies in $\acPic(X^{\an})_{\eqv,\semip,\mQ}$ by \autoref{coro:analytic_invariant_ssa}. 

If $Y$ is a closed subvariety of $X$ such that 
\[
\int_{Y^{\an}} c_1(\overline{L}_{f})_{x_0}^{\dim Y} = \int_{Y_{\sH_{x_{0}}}^{\an}} c_1(F_{x_0}^*\overline{L}_{f})^{\dim Y}\neq 0,
\]
then $Y$ is non-degenerate. 
\end{prop}

\begin{proof}
This follows immediately from \autoref{thm:intersection_Guo_epsilon} and definitions. 
\end{proof}

\begin{remark}
To the author's knowledge, results in the literature proving that a subvariety is non-degenerate have only utilized complex analytic techniques. 
These work use that non-degeneracy can be characterized using the Betti form in the setting of a subvariety of an abelian scheme \cite{DGH:Uniform} or the bifurification measure in the setting of a hypersurface in a polarized dynamical system \cite{DemarcoMavraki:DynamicsP1}. 
In the non-Archimedean, non-trivially valued case, we mention a result of Gubler \cite{Guber:BogomolovConjecture} where he proves a tropical equidistribution result for a totally degenerate abelian variety; here the equilibirum measure is the pushforward to the abelian variety of the probability Haar measure on the canoncial skeleton, which is a real torus.  

When the polarized dynamical system has good reduction at some place or when one works over a trivially valued field, the equilibrium measure is the Dirac measure of a Shilov point on the associated Berkovich space, in particular the Shilov point corresponding to the model witnessing good reduction.  
In this setting, the equilibrium measure is ``small'', and so it is very difficult to deduce non-degeneracy of a subvariety as the measure does not detect much. 
That being said, in recent years, there have been major advances in pluripotential theory over trivially valued fields \cite{BoucksomJonsson:GlobalPluriTriv}, which have had substantial applications (see e.g., \cite{BoucksomJonsson:NAKStability1, BoucksomJonsson:NAKStability2}). 
While we do not have an immediate application in mind, we do hope that \autoref{prop:nondegenerate_trivially} will have applications to non-degeneracy results in the future.
\end{remark}

\section{A global Monge--Amp\`ere measure}
\label{sec:globalMA}
To conclude, we define a notion of a global Monge--Amp\`ere measure on the analytification of a quasi-projective arithmetic variety. 
We note that Song \cite[Section 6]{Song:EquivariantAdelic} has already defined a global Monge--Amp\`ere measure using his results and the normalized boundary of the analytification of a quasi-projective arithmetic variety. The integration pairing from his contruction is induced by an extension of Yuan and Zhang's variant of the Arakelov intersection pairing to the quasi-projective setting \cite[Section 4.1]{YuanZhang:AdelicLineBundles}. 

We will take a different approach and define the global Monge--Amp\`ere measure using families of Monge--Amp\`ere measures we constructed in Subsection \ref{subsec:familiesMA}. 

\subsection{A measure on $\sM(\mZ)$}
We would like to endow $\sM(\mZ)$ with a measure. 
First, we recall that in \autoref{exam:BerkovichoverZ}, we identified the non-Archimedean branches with $[0,+\infty]_p$. We could also identify them with $[0,1]_p$ in which $|\cdot|_{p,\varepsilon}$ is the non-Archimedean absolute value for which $|p|_{p,\varepsilon} = \varepsilon$ for all $\varepsilon\in (0,1)$. 
In fact, this is how Berkovich presented $\sM(\mZ)$ in \cite[Example 1.4.1]{BerkovichSpectral}. Taking limits, we declare that $|\cdot|_{p,0} = |\cdot|_{p,\infty}$ in the previous notation and $|\cdot|_{p,1} = |\cdot|_{p,0} = |\cdot|_0$. 
With this identification, define the branches
\begin{align*}
I_{p} &= \{ |\cdot|_{p,\varepsilon} : \varepsilon \in [0,1]\},\\
I_{\infty} &= \{ |\cdot|_{\infty,\varepsilon} : \varepsilon \in [0,1]\}
\end{align*}
Similar to in \autoref{exam:BerkovichoverZ}, we have that the maps $f_{v}\colon [0,1] \to\sM(\mZ)$ defined by $\varepsilon \mapsto |\cdot|_{p,\varepsilon}$ are homeomorphisms onto $I_v$. 
With this, we now identify the non-Archimedean branches of $\sM(\mZ)$ with $[0,1]_{p}$.

Next, we briefly comment on the topology of $\sM(\mZ)$. 
By definition, $\sM(\mZ)$ carries the weakest topology making the maps $|\cdot| \mapsto |n|$ continuous, and since the absolute values are multiplicative, we have that it suffices to say that the maps $|\cdot|\mapsto |p|$ are continuous for all primes $p$. 
As such, the topology on $\sM(\mZ)$ is generated by the sets
\begin{align*}
U_{p,t}^+ &= \{ |\cdot| : |p| > t\},\\
U_{p,t}^- &= \{ |\cdot| : |p| < t\}. 
\end{align*}
Let $\cB(\sM(\mZ))$ denote the $\sigma$-algebra of Borel sets on $\sM(\mZ)$. 
For later use, we identify an element of $\cB(\sM(\mZ))$ not appearing above, namely
\begin{align*}
A_{p,t_1,t_2} &:= (\sM(\mZ)\setminus U_{p,t_1}^+) \cap (\sM(\mZ)\setminus U_{p,t_2}^-) = \{|\cdot| : t_1 \geq |p| \geq t_2\},
\end{align*}
which we can think of as an annulus type object.  We define our measure as follows:~for $E\in \cB(\sM(\mZ))$, we have that 
\begin{equation}\label{eqn:measuremuprime}
\mu'(E) = \sum_{v\in M_{\mQ}} \frac{\ell(f_{v}^{-1}(E \cap I_v))}{v \log(v)}
\end{equation}
where $\ell(\cdot)$ denote the usual length of a sub-interval in $[0,1]$ and when $v = \infty$, we set $v \log(v) = e$ and when $E \cap I_v = \emptyset$, we set $\ell(\emptyset) = 0$. 

We claim that this is a finite measure. 
Writing 
\[
\sM(\mZ) = U_{p,0}^+ \sqcup U_{p,0}^- \sqcup A_{p,0,0}
\]
we compute that $\mu'(U_{p,0}^-) = \mu'(A_{p,0,0}) = 0$ as $U_{p,0}^- = \emptyset$ and $A_{p,0,0} = x_{p,0}$.
Furthermore
\[
\mu'(U_{p,0}^+) = \pwr{\sum_{p \in M_{\mQ}\setminus \infty} \frac{1}{p\log(p)}} + \frac{1}{e} < \infty
\]
since the summation $\sum_{p}1/p\log(p)$ converges using the prime number theorem. 
Note that by conventions, we have that $\ell((0,1]) = \ell([0,1]) = 1$; this only occurs at $p$ since $f_{p}^{-1}(U_{p,0}^+ \cap I_p) = (0,1]$. 
Therefore, $\mu'$ is a finite measure, and hence we may define a probability measure $\mu$ on $ \sM(\mZ)$ via 
\begin{equation}\label{eqn:measuremu}
\mu \coloneqq \frac{1}{\mu'(\sM(\mZ))}\mu'.
\end{equation}

\subsection{Construction of global Monge--Amp\`ere measure as a pullback measure}
With the measure defined above, we will construct a global Monge--Amp\`ere measure by using a ``pullback'' measure construction.  
More precisely, we define a measure on the Berkovich analytification of a quasi-projective arithmetic variety using the measure $\mu$ on the base $\sM(\mZ)$ and the fact that each fiber over $x \in \sM(\mZ)$ comes equipped with a measure, namely the Monge--Amp\`ere measure constructed in Subsection \ref{subsec:familiesMA}. 

\begin{definition}
Let $\cU$ be a quasi-projective arithmetic variety of dimension $n$, let $\cL_1,\dots,\cL_n$ be semiample line bundles on $\cU$ equipped with strongly semiample metrics $\phi_1,\dots,\phi_n$, and let $\pi\colon \cU^{\an} \to \sM(\mZ)$ denote the Berkovich analytification.

Let $f\in C_c(\cU^{\an})$. 
The \cdef{global Monge--Amp\`ere measure with respect to $\mu$} on $\cU^{\an}$ associated to the the metrized line bundles $\overline{\cL}_1,\dots,\overline{\cL}_n$ is 
\[
\int_{\cU^{\an}} f \, c_1(\overline{\cL}^{\an}_1)\cdots c_1(\overline{\cL}^{\an}_n) \coloneqq \int_{\sM(\mZ)} \left( \int_{\cU_{\sH_x}^{\an}} f_{|\cU_{\sH_x}^{\an}} \, (c_1(\overline{\cL}^{\an}_1)\cdots c_1(\overline{\cL}^{\an}_n))_x\right)\, d\mu
\]
where $(c_1(\overline{\cL}_1^{\an})\cdots c_1(\overline{\cL}_n^{\an}))_x$ denote the Monge--Amp\`ere measures from \autoref{defn:MAmeasure_quasiproj} and $\mu$ is the measure on $\sM(\mZ)$ defined in \eqref{eqn:measuremu}. 
\end{definition}

\begin{example}
Let $\cX$ be a projective arithmetic variety of dimension $n$, let $\cL_1,\dots,\cL_n$ be semiample line bundles on $\cX$ equipped with strongly semiample metrics $\phi_1,\dots,\phi_n$, and let $\pi\colon \cX^{\an} \to \sM(\mZ)$ denote the Berkovich analytification. 
Note that since $\cX \to \Spec(\mZ)$ is flat, the intersection numbers are constant in a family, and hence we
\[
\int_{\cX^{\an}} c_1(\overline{\cL}^{\an}_1)\cdots c_1(\overline{\cL}^{\an}_n)  = \int_{\sM(\mZ)} (\cL_{1|\cX^{\an}_{\sH_x}}^{\an} \cdots \cL_{n|\cX^{\an}_{\sH_x}}^{\an}) \, d\mu = (\cL_{1|\cX_{\sH_x}} \cdots \cL_{n|\cX_{\sH_x}}).
\]
\end{example}

\begin{remark}
\begin{enumerate}
\item[]
\item As with Song's construction from \cite[Section 6]{Song:EquivariantAdelic}, our global Monge--Amp\`ere measure is non-canonical as it depends on the choice of measure on $\sM(\mZ)$. 
\item It would be interesting to compute the total mass of the global Monge--Amp\`ere measure in the quasi-projective setting. To do so, one would need to globalize the arguments of Guo \cite{Guo:Integration}, which would involve developing a theory of forms and currents on Berkovich spaces over general Banach rings (cf.~\cite[Section 0.4.3]{ChambertLoirDucros:FormesDifferentielles}).  
\end{enumerate}
\end{remark}

  \bibliography{refs}{}
\bibliographystyle{amsalpha}

 \end{document}